\documentclass{amsart}
\usepackage{amsmath, amssymb, amscd}
\usepackage{amsthm}
\usepackage{mathrsfs}
\usepackage{mathdots}
\usepackage{bm}
\usepackage[margin=1in]{geometry}
\usepackage{mathtools}

\newcommand{\gll}{{\mathrm{GL}}}
\newcommand{\zz}{{\mathbb Z}}
\newcommand{\pzz}{{\mathbb Z_{>0}}}

\newcommand{\rr}{{\mathbb R}}

\newcommand{\cc}{{\mathbb C}}

\newcommand{\surj}{\twoheadrightarrow}
\newcommand{\map}{\rightarrow}

\newcommand{\ind}{\mathrm{Ind}}
\newcommand{\cind}{\mathrm{ind}}

\newcommand{\Ho}{\mathrm{Hom}}
\newcommand{\wsh}{W_{\mathrm{Sh}}}

\newcommand{\diag}{{\mathrm{diag}}}

\theoremstyle{definition}
\newtheorem{theo}{Theorem}[section]

\newtheorem{prop}[theo]{Proposition}
\newtheorem{lemm}[theo]{Lemma}
\newtheorem{cor}[theo]{Corollary}
\newtheorem{rem}[theo]{Remark}
\newtheorem{conj}[theo]{Conjecture}

\begin{document}

\title{On branching laws of Speh representations}

\author{Nozomi Ito}
\address{Department of Mathematics,
Kyoto University,
Kitashirakawa Oiwake-cho, Sakyo-ku,
Kyoto 606-8502, Japan}
\email{nozomi@math.kyoto-u.ac.jp}
\keywords{zeta integrals, Speh representations, Miyawaki lifts}
\subjclass[2010]{22E50 (primary), 11F70 (secondary)}
\maketitle

\begin{abstract}
In this paper, we consider the branching law of the Speh representation $\mathrm{Sp}(\pi,n+l)$ of $\gll_{2n+2l}$ with respect to the block diagonal subgroup $\gll_n\times\gll_{n+2l}$ for any irreducible generic representation $\pi$ of $\gll_2$ over any $p$-adic field.
We use the Shalika model of $\mathrm{Sp}(\pi,n)$ to construct certain zeta integrals, which were defined by Ginzburg and Kaplan independently, and study them.
Finally, using these zeta integrals, we obtain a  nonzero $\gll_n\times\gll_{n+2l}$-map from  $\mathrm{Sp}(\pi,n+l)$ to $\tau\boxtimes\tau^\vee\chi_\pi\times\mathrm{Sp}(\pi, l)$ for any irreducible representation $\tau$ of $\gll_n$.
These results form part of the local theory of the Miyawaki lifting for unitary groups.
\end{abstract}

\setcounter{tocdepth}{1}
\setcounter{section}{-1} 
\section{Introduction}\label{S:int}
In \cite{ikeda2006pullback}, Ikeda constructed a certain lifting of Siegel modular forms by using the pullbacks of Ikeda lifts to block diagonal subsets as kernel functions to approach Miyawaki's conjecture \cite{miya}.
It is called the Miyawaki lifting today and its analogues for other groups and representation-theoretical generalizations were given by some researchers (\cite{ATOBE2018281}, \cite{Kim:2018tm}, \cite{Atobe:2020uz} and others).
Representation-theoretically, Ikeda lifts are characterized as automorphic representations whose unramified components are isomorphic to some (quotients of) degenerate principal series representations induced from Siegel parabolic subgroups.
Then, the Miyawaki lifting is a construction of automorphic representations using the pullbacks of automorphic forms in Ikeda lifts to diagonal subgroups (e.g. $\mathrm{Sp}(W_1)\times\mathrm{Sp}(W_2)$ for $\mathrm{Sp}(W_1\oplus W_2)$, where $W_1,W_2$ are symplectic spaces)  as kernel functions.
Thus, the study of the local theory of Miyawaki lifting is to study the branching laws of local components of Ikeda lifts with respect to diagonal subgroups.

Now we consider the local theory of the Miyawaki lifting for unitary groups (see \cite{itofirst} for more details) at finite places where the unitary groups split.
In this case, what we study is the branching law of any irreducible representation $\Pi'$ of $\gll_{2n+2l}$ which is realized as the unique irreducible quotient of \mbox{$\pi'|\cdot|^{(n+l-1)/2}\times\dots\times\pi'|\cdot|^{-(n+l-1)/2}$} (normalized induction) for some generic unitary representation $\pi'$ of $\gll_2$ with respect to the diagonal subgroup $\gll_n\times\gll_{n+2l}=\diag(\gll_n,\gll_{n+2l})\subset\gll_{2n+2l}$, where $n,l\in\zz_{\geq 0}$, $F$ is a $p$-adic field with absolute value $|\cdot|$, and we write $\gll_m=\gll_m(F)$ for short.
In general, $\Pi'$ is isomorphic to the Speh representation $\mathrm {Sp}(\pi',n+l)$ (see \S \ref{SS:sp}).
If $\pi'=\chi_1\times\chi_2$ for some characters $\chi_1, \chi_2$ of $F^\times$, then $\mathrm {Sp}(\pi',n+l)\simeq\chi_1(\det)\times\chi_2(\det)$ and the branching law of $\chi_1(\det)\times\chi_2(\det)|_{\gll_n\times\gll_{n+2l}}$ was studied  in the context of the doubling method.
However, we cannot use the same method if $\pi'$ is a supercuspidal representation.
Moreover, since the doubling method is a theory for L-functions of $G\times\gll_1$ for some algebraic group $G$, it  could be incompatible with the Miyawaki lifting, which is expected to be related to some L-function of $G\times\gll_2$, when considering the global theory.
Thus we need a new theory of zeta integrals.

What we focus on is the recent paper of Lapid and Mao \cite{lapid_mao_2020}.
They constructed analogues of Rankin-Selberg integrals for Speh representations by using the Shalika models of Speh representations.
Based on this result, we construct a modification of doubling zeta integrals  (to be mentioned later, it turned out that these integrals had already been defined by other researchers).

Let us state the main results of this paper.
Let $\pi$ be a generic irreducible representation of $\gll_2$ with central character $\chi_\pi$.
We assume that $\pi$ is approximately tempered (see \S \ref{SS:not}, note that if $\pi$ is unitary, then it is approximately tempered).
For $m\in\zz_{>0}$, we denote the Speh representation $\mathrm{Sp}(\pi, m)$ by $\pi_m$ for short.
For a representation $\mu$, we denote the (smooth) dual of $\mu$ by $\mu^\vee$.

First we consider the equal rank case, i.e., $l=0$.
We fix an additive character $\psi$ of $F$ and denote the Shalika model  of $\pi_n$, i.e., a (unique) realization of $\pi_n$ in $\ind^{\gll_{2n}}_U\psi\circ\mathrm{tr}$ for $U={\{(
\begin{smallmatrix}
1_n&*\\
0_n&1_n
\end{smallmatrix})\}}\simeq \mathrm M_n(F)$ (from \S \ref{SS:mod}, we use a slightly different notation),  by $\mathcal W^{\psi}_{\mathrm{Sh}}(\pi_n)$.
Let $\tau\neq0$ be a smooth representation of $\gll_n$ such that $\tau\subset \tau_1\times\cdots\times\tau_m$ for some irreducible representations $\tau_i$.
For $\wsh\in \mathcal W^\psi_{\mathrm{Sh}}(\pi_n)$, $s\in\cc$, and a matrix coefficient $f$ of $\tau$, we put
$$Z(\wsh,s,f):=\int_{\gll_n}\Phi_{\wsh} (g)f(g)|\det g|^{s-\frac{1}{2}}dg,$$
where $\Phi_{\wsh}$ is the restriction of $\wsh$ to $\diag(\gll_n,1_n)\simeq\gll_n$.
Then the following holds:
\begin{theo}[Theorem \ref{T:main1}]\label{T:int1}
\begin{enumerate}
\renewcommand{\labelenumi}{(\roman{enumi})}
\item If $\mathrm{Re}(s)$ is sufficiently large, then the integral defining $Z(\wsh,s,f)$ converges absolutely for any $\wsh$ and $f$.
Moreover, $Z(\wsh,s,f)$ admits meromorphic continuation to all of $\cc$ and there is a (unique) polynomial $P(X)\in\cc[X]$ such that $P(0)=1$ and 
$$\langle Z(\wsh,s,f)\ | \ \wsh\in \mathcal W^\psi_{\mathrm{Sh}}(\pi_n), \ f: \text{a matrix coefficient of } \tau \rangle_\cc=P(q^{-s})^{-1}\cc[q^{-s},q^{s}],$$
where $q$ is the cardinality of the residue field of $F$; denote $P(q^{-s})^{-1}$  by $L(\pi;s,\tau)$.
\item  There is a function $\gamma(s)\in\cc(q^{-s})$ such that 
$$Z(\widetilde\wsh,1-s,f(\cdot^{-1}))=\gamma(s)Z(\wsh,s,f)$$
for any $\wsh$ and $f$, where $\widetilde\wsh:=\chi_\pi^{-1}(\det)\wsh(\left(\begin{smallmatrix}1_n&0\\0&-1_n\end{smallmatrix}\right)\cdot\left(\begin{smallmatrix}0&1_n\\-1_n&0\end{smallmatrix}\right))(\in \mathcal W^{\psi^{-1}}_{\mathrm{Sh}}(\mathrm{Sp}(\pi^\vee,n))$; denote $\gamma(s)$  by $\gamma(\pi;s,\tau,\psi)$.
\item If $\tau$ is irreducible and generic, then
$$L(\pi;s,\tau)=L(s,\pi\boxtimes\tau),$$
where the right-hand side is the local L-factor of $\pi\boxtimes\tau$ defined by Jacquet, Piatetski-Shapiro and
Shalika \cite{JPSS}.
\end{enumerate}
\end{theo}
\begin{rem}\label{R:sadsad}
\begin{itemize}
\item We constructed the above zeta integral based on the work of Lapid and Mao \cite{lapid_mao_2020}, but in fact the integral had already been defined by Ginzburg (\cite{gin}) and Kaplan (Appendix C of \cite{cfk}) in a more general setting, namely for an irreducible generic representation $\pi$ of $\gll_k(F)$ for any $k$ and any local field $F$ of characteristic zero (either archimedean or $p$-adic), independently.
Then, the really new result in the above theorem is (iii) only
(see Remark \ref{R:sad}).
\item In recent years, some analogues of Rankin-Selberg zeta integrals using Speh representations were given by some researchers:
\begin{itemize}
\item Zeta integrals for Speh representations of type $(n,k)\times(n,k)$ were defined in  \cite{lapid_mao_2020}.
\item We can find zeta integrals for Speh representations of type $(n,k)\times(n-1,k)$ in the recent work of Atobe, Kondo, and Yasuda (\cite{aky}).
\item The zeta integrals defined in Appendix C of \cite{cfk} (and  \cite{gin}) are (essentially) for Speh representations of type $(n,1)\times(k,n)$.
\end{itemize}
Here, we say that the Speh representation $\mathrm {Sp}(\pi',l)$ is of type  $(m,l)$ if $\pi'$ is a (generic, irreducible) representation of $\gll_m$.
Then, in all three of the above papers, the determination of the L-factors was left unsolved (some partial results, such as  \cite[Proposition C.10]{cfk}, were given, see Remark \ref{R:sad}).
In contrast, our result (iii) is a fortunate example solving this problem, albeit only for Speh representations of type $(n,1)\times(2,n)$.
\end{itemize}
\end{rem}

Next we consider the general rank case.
Since 
$$\Phi_{\pi_n(\diag(g_1,g_2))\wsh}=\Phi_{\wsh} (g_2^{-1}\cdot g_1)\chi_\pi(\det  g_2)$$
 for any $\wsh\in \mathcal W^\psi_{\mathrm{Sh}}(\pi_n)$ and $g_1,g_2\in\gll_n$, the linear extension of $L(\pi;s,\tau^\vee)^{-1}Z(\cdot,s,\cdot)|_{s=\frac{1}{2}}$ on $W_{\mathrm{Sh}}(\pi_n)\otimes\tau^\vee\otimes\tau$ defines a nonzero element of 
 $$\Ho_{\gll_n\times \gll_n}(\pi_n\otimes(\tau^\vee\boxtimes\tau\chi_\pi^{-1}),\cc)\simeq\Ho_{\gll_n\times \gll_n}(\pi_n,\tau\boxtimes\tau^\vee\chi_\pi).$$
Then, by simple consideration (see \S\ref{S:grank}), the following holds:
\begin{theo}[Theorem \ref{T:main2}]\label{T:int2}
The space
$$\Ho_{\gll_n\times\gll_{n+2l}}(\pi_{n+l},\tau\boxtimes\tau^\vee\chi_\pi\times\pi_l)$$
is nonzero.
\end{theo}
\begin{rem}
The above theorem only means that local Miyawaki lifts for split unitary groups are always nonvanishing.
To complete our purpose, there remain two problems, namely, uniqueness and multiplicity at most one (see Conjecture \ref{C:end}).
\end{rem}
We now give the organization of this paper.
In \S\ref{S:pre}, we introduce the notations, Speh representations, and the models of Speh representations according to \cite{lapid_mao_2020}.
In \S\ref{S:zeta},  we introduce and study the above zeta integral.
We show some properties which zeta integrals should have in general, and give the proof of Theorem \ref{T:int1}.
However, we postpone the proof of the functional equation, which is necessary to show Theorem \ref{T:int1}(iii), to the next section.
In \S\ref{S:fe}, we show the functional equation as just announced.
The essential of the section is the inequality $\dim_\cc\Ho_{\gll_n\times\gll_n}(\pi_n,\tau\boxtimes\tau^\vee\chi_\pi)\leq 1$ for supercuspidal $\tau$. 
In \S\ref{S:grank}, we make some remarks about the branching laws of Speh representations with respect to block diagonal subgroups of general size, including the proof of Theorem \ref{T:int2}.

\subsection*{Acknowledgment}
The author would like to thank his adviser Prof. Atsushi Ichino for his helpful advice.
He also would like to thank  Prof. Yuanqing Cai and  Prof. Eyal Kaplan for their pointing out the overlap of our work with some previous studies and for their helpful comments on this paper.

This work was supported by JSPS Research Fellowships for Young Scientists KAKENHI Grant Number 20J10875.

\section{Preliminaries}\label{S:pre}
\subsection{Notation}\label{SS:not}
Throughout this paper, fix a $p$-adic field $F$ with absolute value $|\cdot|$ and ring of integers $\mathcal O$.
Let $q$ be the cardinality of the residue field of $F$.
If $G$ is an algebraic group over $F$, we also use $G$ to denote $G(F)$.
The term `representation' is used to refer to a smooth, complex representation of an algebraic group over $F$.

We denote by $\mathrm{Irr}\gll_m$ the set of equivalence classes of irreducible representations of $\gll_m$ and put 
$\mathrm{Irr}=\bigcup_{0\leq m}\mathrm{Irr}\gll_m$.
We denote by $\mathrm{Irr_{\rm gen}}\gll_m$ (resp. $\mathrm{Irr_{\rm sc}}\gll_m$) the subset consisting of all generic (resp. supercuspidal)  elements  of $\mathrm{Irr}\gll_m$ and put $\mathrm{Irr_{\rm gen}}=\bigcup_{0\leq m}\mathrm{Irr_{\rm gen}}\gll_m$, $\mathrm{Irr_{\rm sc}}=\bigcup_{0\leq m}\mathrm{Irr_{\rm sc}}\gll_m$.

For $\bm m = (m_1,\dots, m_l)\in(\zz_{>0})^l$, we denote the block upper triangular parabolic subgroup of type $\bm m$ with unipotent radical $U_{\bm m}$ by $P_{\bm m}$, a Levi subgroup $\diag(\gll_{m_1},\dots, \gll_{m_l})\simeq \gll_{m_1}\times\dots\times \gll_{m_l}$ of $P_{\bm m}$ by $M_{\bm m}$, and the modulus character of $P_{\bm m}$ by $\delta_{P_{\bm m}}$.

We use the notation $\ind^G_H$ and $\cind^G_H$ to denote induction and induction with compact support (both {\it unnormalized}) from a subgroup $H$ of $G$.
If $\pi_1,\dots,\pi_l$ are representations of $\gll_{m_1},...,\gll_{m_l}$ respectively, then we denote the parabolically induced representations 
\begin{align*}
\ind^{\gll_{m}}_{P_{\bm m}}\delta_{P_{\bm m}}^{\frac{1}{2}}\otimes\pi_1\boxtimes\dots\boxtimes\pi_l  \text{ and } \ind^{\gll_{m}}_{P_{\bm m}}\pi_1\boxtimes\dots\boxtimes\pi_l
\end{align*}
by $\pi_1\times\dots\times\pi_l$ (normalized induction) and $\pi_1*\dots*\pi_l$ (unnormalized induction), respectively, where $m=\sum_i m_i$ and  \mbox{$\bm m=(m_1,\dots,m_l)$}.

If there is no confusion, we often denote $(\overbrace{m_1,m_1,\dots,m_1}^{l_1},\overbrace{m_2,m_2,\dots,m_2}^{l_2},\dots)$ by $({m_1}^{l_1},{m_2}^{l_2},\dots)$.

Let $\pi$ be a representation of $\gll_m$.
We denote the contragradient representation of $\pi$ by $\pi^\vee$ and $\pi\otimes\chi\circ\det$ by $\pi \chi$ for any $\chi\in \mathrm{Irr}\gll_1$.

For $\pi\in{\rm Irr_{\rm sc}}$, we denote the unique irreducible subrepresentation of $\pi|\cdot|^{\frac{m-1}{2}}\times\pi|\cdot|^{\frac{m-3}{2}}\times\dots\times\pi|\cdot|^{-\frac{m-1}{2}}$
by ${\rm St}(\pi,m)$ (generalized Steinberg representation).

We denote by $\mathrm{Alg'}\gll_m$ the set of equivalence classes of representations $\pi\neq 0$ of $\gll_m$ such that 
$$\pi\subset \pi_1\times\dots\times\pi_l$$
 for some $\pi_1,\dots,\pi_l\in{\rm Irr}$ (equivalently, $\pi_1,\dots,\pi_l\in{\rm Irr_{\rm sc}}$) and put $\mathrm{Alg'}=\bigcup_{0\leq m}\mathrm{Alg'}\gll_m$.
We note that $\mathrm{Alg'}$ is closed under parabolic induction.
For any $\pi\in\mathrm{Alg'}$, we denote the central character of $\pi$ by $\chi_\pi$.

Let $\pi\in\mathrm{Irr_{\rm gen}}$.
Then, $\pi$ can be written uniquely (up to permutation) as 
$$\pi=\mathrm{St}(\rho_1,m_1)|\cdot|^{r_1}\times\dots\times\mathrm{St}(\rho_{l},m_{l})|\cdot|^{r_{l}}$$
with cuspidal unitary representations $\rho_i$ and $r_i\in\rr$.
We say that $\pi$ is approximately tempered  if $r_i-r_j<1$ for any $i$ and $j$, following \cite{lapid_mao_2020}.
We note that if $\pi\in\mathrm{Irr_{\rm gen}}$ is essentially unitary, then it is approximately tempered.

\subsection{Speh representations (cf. \cite[\S2]{lapid_mao_2020})}\label{SS:sp}
For the rest of this section, fix $\pi\in\mathrm{Irr_{\rm gen}}\gll_k$.

Let $\{\rho_1,\dots,\rho_r\} \ (\rho_i \in {\rm Irr_{sc}})$ be the cuspidal support of $\pi$, which is a multiset of $\mathrm{Irr_{\rm sc}}$.
For each $n\in\zz_{>0}$, we define the Speh representation $\mathrm{Sp}(\pi, n)$ as the representation corresponding to the multisegment 
$$\sum_{i=1}^r\{\rho_i|\cdot|^{-\frac{n-1}{2}},\dots,\rho_i|\cdot|^{\frac{n-1}{2}}\}$$
under the Zelevinsky classification \cite{zel}.
By rearranging the indices, we assume that $\rho_1,\dots,\rho_{i-1}\neq \rho_i |\cdot|^{-m}$ for all $i$ and all $m\in\zz_{>0}$.
Then, $\mathrm{Sp}(\pi, n)$ is the unique irreducible subrepresentation of 
$$\mathrm{Sp}(\rho_1,n)\times\dots\times \mathrm{Sp}(\rho_r,n),$$
where $\mathrm{Sp}(\rho_i,n)$ is the unique irreducible quotient of  $\rho_i|\cdot|^{\frac{n-1}{2}}\times\rho_i|\cdot|^{\frac{n-3}{2}}\times\dots\times \rho_i|\cdot|^{-\frac{n-1}{2}}$.

Other realizations of $\mathrm{Sp}(\pi, n)$ are known:

\begin{prop}[{\cite[Corollary 2.11]{lapid_mao_2020}}]\label{P:sp}
$\mathrm{Sp}(\pi, n)$ is a unique irreducible subrepresentation of
$$\mathrm{Sp}(\pi, n-1)|\cdot|^{-\frac{1}{2}}\times \pi|\cdot|^{\frac{n-1}{2}}.$$
In particular,
$\mathrm{Sp}(\pi, n)$ is both a subrepresentation of
$$\Pi=\pi|\cdot|^{-\frac{n-1}{2}}\times\pi|\cdot|^{-\frac{n-3}{2}}\times\dots\times \pi|\cdot|^{\frac{n-1}{2}}$$
and a quotient of
$$\tilde\Pi=\pi|\cdot|^{\frac{n-1}{2}}\times\pi|\cdot|^{\frac{n-3}{2}}\times\dots\times \pi|\cdot|^{-\frac{n-1}{2}}.$$
\end{prop}
\begin{rem}[{{\cite[Remark 2.12]{lapid_mao_2020}}}]\label{R:sp}
Assume $\pi$ is approximately tempered.
Then, $\mathrm{Sp}(\pi, n)$ is the Langlands quotient of $\tilde\Pi$.
In particular, $\mathrm{Sp}(\pi, n)$ is the unique subrepresentation of $\mathrm{Sp}(\pi, n_1)|\cdot|^{-\frac{n_2}{2}}\times\mathrm{Sp}(\pi, n_2)|\cdot|^{\frac{n_1}{2}}$ for any $n_1,n_2\in \pzz$ such that $n_1+n_2=n$.
\end{rem}
From now on, we will denote $\mathrm{Sp}(\pi, n)$ by $\pi_n$ for short.

\subsection{The models (cf. \cite[\S3]{lapid_mao_2020})}\label{SS:mod}
For the rest of this section, fix a nontrivial character $\psi$ of $F$.

For each $n\in\zz_{>0}$, we define $w_{(k,n)}\in\gll_{kn}$ by
$$(w_{(k,n)})_{i,j}=\delta_{k(i-pn+n-1)+p, j} \ \text{ if } pn-n+1\leq i \leq pn$$
for $p\in\{1,\dots,k\}$ and  a function $\Psi_{(k,n)}:\gll_{kn} \map \cc$ by 
$$\Psi_{(k,n)}(g)=\psi\left(\sum_{k  \nmid i}  g_{i,i+1}\right).$$
Then, the restrictions of $\Psi_{(k,n)}$ to $U_{(1^{kn})}$ and $U_{(n^k)}^{w_{(k,n)}}=w_{(k,n)}^{-1}U_{(n^k)} w_{(k,n)}$ are both characters.
We note that  
$$\Psi_{(k,n)}\left(
\begin{pmatrix}
1_n&X_1&&&\\
&1_n&X_2&&\raisebox{-6pt}[0pt][0pt]{\makebox[-10pt][r]{\Huge *}}\\
&&\ddots&\ddots&\\
&&&1_n&X_{k-1}\\
&&&&1_n
\end{pmatrix}
^{w_{(k,n)}}\right)=\psi\left(\sum_{1\leq i \leq k-1}{\rm tr}X_i\right),$$
where $X_i\in {\rm M}_k(F)$.
We put
$$\mathcal W^\psi_{\mathrm{Ze},k,n}=\ind^{\gll_{kn}}_{U_{(1^{kn})}}\Psi_{(k,n)}|_{U_{(1^{kn})}} \ \text{ and } \ \mathcal W^\psi_{\mathrm{Sh},k,n}=\ind^{\gll_{kn}}_{U_{(n^k)}^{w_{(k,n)}}}\Psi_{(k,n)}|_{U_{(n^k)}^{w_{(k,n)}}}.$$
Then, it is known that $\dim_\cc(\pi_n, \mathcal W^\psi_{\mathrm{Ze},k,n})=\dim_\cc(\pi_n, \mathcal W^\psi_{\mathrm{Sh},k,n})=1$ (\cite[Theorem 3.1]{lapid_mao_2020}) .
We denote the images of $\pi_n$ on $\mathcal W^\psi_{\mathrm{Ze},k,n}$ and $\mathcal W^\psi_{\mathrm{Sh},k,n}$ by $\mathcal W^\psi_{\mathrm{Ze}}(\pi_n)$ (Zelevinsky model) and $\mathcal W^\psi_{\mathrm{Sh}}(\pi_n)$ (Shalika model), respectively.
We note that
$$\wsh(\diag(\overbrace{g,\dots,g}^{k})^{w_{(k,n)}} \cdot )=\chi_\pi(\det g)\wsh, \ g\in \gll_n$$
 for any $\wsh\in \mathcal W^\psi_{\mathrm{Sh}}(\pi_n)$.
 
The relation between $\mathcal W^\psi_{\mathrm{Ze}}(\pi_n)$ and $\mathcal W^\psi_{\mathrm{Sh}}(\pi_n)$ is as follows:
\begin{prop}[{{\cite[Lemma 3.8, 3.11]{lapid_mao_2020}}}]\label{P:zs}
Let $\mathfrak n_n$ (resp. $\bar{\mathfrak n}_n=\!^t{\mathfrak n}_n$) be the set of  upper (resp. lower) triangular nilpotent matrices in $\mathrm M_n(F)$ and
$$N_n=
w_{(k,n)}^{-1}\begin{pmatrix}
1_n& \bar{\mathfrak n}_n&\bar{\mathfrak n}_n&\dots&\bar{\mathfrak n}_n\\
&1_n&\bar{\mathfrak n}_n&\dots&\bar{\mathfrak n}_n\\
&&\ddots&\ddots&\vdots\\
&&&1_n&\bar{\mathfrak n}_n\\
&&&&1_n
\end{pmatrix}w_{(k,n)}, \
N'_n=1_{kn}+
w_{(k,n)}^{-1}\begin{pmatrix}
 {\mathfrak n}_n&&&&\\
{\mathfrak n}_n&{\mathfrak n}_n&&&\\
\vdots&\vdots&\ddots&&\\
{\mathfrak n}_n&{\mathfrak n}_n&\dots&{\mathfrak n}_n&\\
0_n&0_n&\dots&\dots&0_n
\end{pmatrix}w_{(k,n)}
.$$
Then, the isomorphism $\mathcal T_n: \mathcal W^\psi_{\mathrm{Ze}}(\pi_n)\overset{\sim}{\map}\mathcal W^\psi_{\mathrm{Sh}}(\pi_n)$ and its inverse  $\mathcal T_n^{-1}: \mathcal W^\psi_{\mathrm{Sh}}(\pi_n)\overset{\sim}{\map}\mathcal W^\psi_{\mathrm{Ze}}(\pi_n)$ are given by
\begin{align*}
\mathcal T_n W_{\rm Ze}=\int_{N_n}W_{\rm Ze}(u\cdot)du, \ \mathcal T_n^{-1} W_{\rm Sh}=\int_{N'_n}W_{\rm Sh}(u'\cdot)du' \end{align*}
for $W_{\rm Ze} \in\mathcal W^\psi_{\mathrm{Ze}}(\pi_n)$ and $\wsh\in \mathcal W^\psi_{\mathrm{Sh}}(\pi_n)$, where these integrands are (pointwise) compactly supported.
\end{prop}

Now we consider `intermediate' models between $\mathcal W^\psi_{\mathrm{Ze}}(\pi_n)$ and $\mathcal W^\psi_{\mathrm{Sh}}(\pi_n)$.
Take
\begin{align*}
\bm \lambda=(\lambda_1,\dots,\lambda_{nk})=(\overbrace{\overbrace{k-1,k-2,\dots,0},\overbrace{k-1,k-2,\dots,0},\dots,\overbrace{k-1,k-2,\dots,0}}^{n})+(\lambda_k^k,\lambda_{2k}^k,\dots,\lambda_{nk}^k)\in\zz^{nk}
\end{align*}
and define a parabolic subgroup $P=MU$ of $\gll_{kn}$ by
$\large P=\left\{ g\in \mathrm \gll_{kn}\ \middle | \ g_{i,j}=0 \text{ if } \lambda_i < \lambda_j\right\}$.
Then, the restriction of $\Psi_{(k,n)}$ to $U$ is a character and
$\dim_\cc(\pi_n, \ind^{\gll_{nk}}_{U}\Psi_{(k,n)}|_U)=1$ in general (\cite[Theorem 3.1]{lapid_mao_2020}). 
Take $\bm n = (n_1,\dots ,n_l)\in (\zz_{>0})^l$ such that $\sum_in_i=n$ and assume that 
$$\lambda_{ik}=-jk \ \text{ if } \  \sum_{p=1}^jn_p< i\leq\sum_{p=1}^{j+1}n_p$$
for $j\in\{0,\dots,l-1\}$.
Then, we have 

$$P=\diag(P_{(n_1^k)}^{w_{(k,n_1)}},\dots,P_{(n_l^k)}^{w_{(k,n_l)}})U_{\bm n}$$
and
$$\Psi|_U=(\Psi_{(k,n_1)}|_{U_{(n_1^k)}^{w_{(k,n_1)}}}\boxtimes\dots\boxtimes \Psi_{(k,n_l)}|_{U_{(n_l^k)}^{w_{(k,n_l)}}})\otimes1_{U_{\bm n}}.$$
For this $\bm n$, we put
$$\mathcal W^\psi_{\mathrm k,\bm n}=\ind^{\gll_{kn}}_{U}\Psi_{(k,n)}|_{U}\simeq \mathcal W^\psi_{\mathrm{Sh},k,n_1}*\dots*\mathcal W^\psi_{\mathrm{Sh},k,n_l}$$
and denote the image of $\pi_n$ on $\mathcal W^\psi_{k,\bm n}$ by $\mathcal W^\psi_{\bm n}(\pi_n)$.
Then, using Proposition \ref{P:zs}, we obtain the relations between $\mathcal W^\psi_{\bm n}(\pi_n)$ and $\mathcal W^\psi_{\mathrm{Ze}}(\pi_n)$, $\mathcal W^\psi_{\mathrm{Sh}}(\pi_n)$ as follows:
\begin{cor}\label{C:im}
Let $\mathfrak n_{\bm n}=U_{\bm n} -1_n, \bar{\mathfrak n}_{\bm n}=\!^t{\mathfrak n}_{\bm n}$ and
$$N_{\bm n,1}=\diag (N_{n_1},\dots,N_{n_l}), \ N'_{\bm n,1}=\diag (N'_{n_1},\dots,N'_{n_l}),$$
 $$
 N_{\bm n,2}=w_{(k,n)}^{-1}\begin{pmatrix}
1_n& \bar{\mathfrak n}_{\bm n}&\bar{\mathfrak n}_{\bm n}&\dots&\bar{\mathfrak n}_{\bm n}\\
&1_n&\bar{\mathfrak n}_{\bm n}&\dots&\bar{\mathfrak n}_{\bm n}\\
&&\ddots&\ddots&\vdots\\
&&&1_n&\bar{\mathfrak n}_{\bm n}\\
&&&&1_n
\end{pmatrix}w_{(k,n)},  \
N'_{\bm n,2}=1_{kn}+
w_{(k,n)}^{-1}\begin{pmatrix}
 {\mathfrak n}_{\bm n}&&&&\\
{\mathfrak n}_{\bm n}&{\mathfrak n}_{\bm n}&&&\\
\vdots&\vdots&\ddots&&\\
{\mathfrak n}_{\bm n}&{\mathfrak n}_{\bm n}&\dots&{\mathfrak n}_{\bm n}&\\
0_n&0_n&\dots&\dots&0_n
\end{pmatrix}w_{(k,n)}.
$$ 
Then, the isomorphisms 
$$\mathcal T_{\bm n,1}: \mathcal W^\psi_{\mathrm{Ze}}(\pi_n)\overset{\sim}{\map}\mathcal W^\psi_{\bm n}(\pi_n),\ \mathcal T_{\bm n,2}: \mathcal W^\psi_{\bm n}(\pi_n)\overset{\sim}{\map}\mathcal W^\psi_{\mathrm{\rm Sh}}(\pi_n)$$ and their inverses
$$\mathcal T_{\bm n,1}^{-1}: \mathcal W^\psi_{\bm n}(\pi_n)\overset{\sim}{\map}\mathcal W^\psi_{\mathrm{Ze}}(\pi_n), \ \mathcal T_{\bm n,2}^{-1}: \mathcal W^\psi_{\mathrm{Sh}}(\pi_n)\overset{\sim}{\map}\mathcal W^\psi_{\bm n}(\pi_n)$$
 are given by
\begin{align*}
\mathcal T_{\bm n,1} W_{\rm Ze}=\int_{N_{\bm n,1}}W_{\rm Ze}(u\cdot)du, \ \mathcal T_{\bm n,2} W=\int_{N_{\bm n,2}}W(u\cdot)du,\\
 \mathcal T_{\bm n,1}^{-1} W=\int_{N'_{\bm n, 1}}W(u'\cdot)du', \ \mathcal T_{\bm n,2}^{-1} W_{\rm Sh}=\int_{N'_{\bm n,2}}W_{\rm Sh}(u'\cdot)du'
\end{align*}
for $W_{\rm Ze} \in\mathcal W^\psi_{\mathrm{Ze}}(\pi_n)$, $W\in \mathcal W^\psi_{\bm n}(\pi_n)$, and $\wsh\in \mathcal W^\psi_{\mathrm{Sh}}(\pi_n)$.
\end{cor}

\begin{rem}\label{R:div}
By  proposition \ref{P:sp}, we have 
$$W_{\rm Ze}|_{M_{(k(n-1),k)}}\in \mathcal W_{\rm Ze}^\psi(\pi_{n-1})|\cdot|^{\frac{1}{2}(k-1)}\otimes\mathcal W_{\rm Ze}^\psi(\pi)|\cdot|^{-\frac{1}{2}(n-1)(k-1)}$$
and
$$W_{\rm Ze}|_{M_{(k^n)}}\in \mathcal W_{\rm Ze}^\psi(\pi)|\cdot|^{\frac{1}{2}(n-1)(k-1)}\otimes\mathcal W_{\rm Ze}^\psi(\pi)|\cdot|^{\frac{1}{2}(n-3)(k-1)}\otimes\dots\otimes\mathcal W_{\rm Ze}^\psi(\pi)|\cdot|^{-\frac{1}{2}(n-1)(k-1)}$$
for any $W_{\rm Ze}\in \mathcal W^\psi_{\mathrm{Ze}}(\pi_n)$ (see \cite[\S 3.1]{lapid_mao_2020}).
Similarly, by Remark \ref{R:sp}, we have
$$W|_{M_{(kn_1,kn_2)}}\in \mathcal \wsh^\psi(\pi_{n_1})|\cdot|^{\frac{1}{2}n_2(k-1)}\otimes\mathcal \wsh^\psi(\pi_{n_2})|\cdot|^{-\frac{1}{2}n_1(k-1)}$$
for any $W\in \mathcal W^\psi_{(n_1,n_2)}(\pi_n)$ if $\pi$ is approximately tempered, where $n_1+n_2=n$.
However, we do not know whether this holds in general.
\end{rem}

\section{The local zeta integral}\label{S:zeta}
For the rest of this paper, fix $\pi\in\mathrm{Irr_{\rm gen}}\gll_2$, $n\in\zz_{>0}$ and a nontrivial character $\psi$ of $F$.
Moreover, for each $m\in\zz_{>0}$, we write $w_{(2,m)}=w_m, \ \mathcal W^\psi_{\mathrm{Ze},2,m}=\mathcal W^\psi_{\mathrm{Ze},m},\  \mathcal W^\psi_{\mathrm{Sh},2,m}=\mathcal W^\psi_{\mathrm{Sh},m}, $ and $\mathcal W^\psi_{ 2,\bm m}=\mathcal W^\psi_{ \bm m}$ for short.
 
\subsection{The main results}\label{SS:main}
For any $\wsh \in \mathcal W^\psi_{\mathrm{Sh}}(\pi_n)$, define a function $\Phi_{\wsh}$ on $\gll_n$ by
$$\Phi_{\wsh}(g)= \wsh(\left(\begin{smallmatrix}g&0\\0&1_n\end{smallmatrix}\right)^{w_n}).$$
We note that 
$$\Phi_{\pi_n(\diag (g_1,g_2)^{w_n})\wsh}=\Phi_{\wsh}(g_2^{-1}\cdot g_1)\chi_2(\det g_2), \ 
\Phi_{\pi_n\left(\left(\begin{smallmatrix}
1_n&X\\
&1_n
\end{smallmatrix}
\right)^{w_n}\right)\wsh}=\psi({\rm tr}(X\cdot))\Phi_{\wsh}$$
for $g_1,g_2\in\gll_n,$ and $X\in \mathrm M_n(F)$ (see \S\ref{SS:mod}).
In particular, $\Phi_{\wsh}$ is bi-$K$-invariant for some open compact subgroup $K$ of $\gll_n$.
Moreover, the set $\{\Phi_{\wsh} \ | \  \wsh\in \mathcal W^\psi_{\mathrm{Sh}}(\pi_n) \}$ does not depend on $\psi$ since for any additive character $\psi'=\psi(a\cdot)$ of $F$, where $a\in F^\times$,
the isomorphism $\mathcal W^{\psi'}_{\mathrm{Sh}}(\pi_n)\overset{\sim}{\map}\mathcal W^{\psi}_{\mathrm{Sh}}(\pi_n)$ is given by 
$$\wsh'\mapsto \wsh'(\diag(1_n,a1_n)^{w_n}\cdot)$$
for $\wsh'\in\mathcal W^{\psi'}_{\mathrm{Sh}}(\pi_n)$.

For $\wsh \in \mathcal W^\psi_{\mathrm{Sh}}(\pi_n)$, define $\widetilde \wsh\in\mathcal W^{\psi^{-1}}_{\mathrm{Sh}}(\pi_n^\vee)$ by
$$\widetilde \wsh(g^{w_n})=\chi_\pi^{-1}(\det g)\wsh(\left(\left(\begin{smallmatrix}1_n&0\\0&-1_n\end{smallmatrix}\right)g\left(\begin{smallmatrix}0&1_n\\-1_n&0\end{smallmatrix}\right)\right)^{w_n})$$
(note that $\pi_n\chi_\pi^{-1}\simeq{\rm Sp}(\pi\chi_\pi^{-1},n)\simeq{\rm Sp}(\pi^\vee,n)\simeq{\rm Sp}(\pi,n)^\vee=\pi_n^\vee$).
We note that 
$$\Phi_{\widetilde\wsh}(g)= \wsh(\left(\left(\begin{smallmatrix}g&0\\0&1_n\end{smallmatrix}\right)\left(\begin{smallmatrix}0&1_n\\1_n&0\end{smallmatrix}\right)\right)^{w_n}).$$

Let $\tau\in {\rm Alg'}\gll_n$. 
For any matrix coefficient $f$ of $\tau$, denote a matrix coefficient  $f(\cdot^{-1})$ of $\tau^\vee$ by $f^\vee$.

Finally, for $\wsh \in \mathcal W^\psi_{\mathrm{Sh}}(\pi_n)$ and a matrix coefficient $f$ of any $\tau\in {\rm Alg'}\gll_n$, define the zeta integral $Z(\wsh,s,f)$ with complex variable $s$
by
$$Z(\wsh,s,f)=\int_{\gll_n}\Phi_{\wsh} (g)f(g)|\det g|^{s-\frac{1}{2}}dg.$$

The main result of this paper is as follows:
\begin{theo}\label{T:main1}
Let $\tau\in {\rm Alg'}\gll_n$.
Assume $\pi$ is approximately tempered except (i).
\begin{enumerate}
\renewcommand{\labelenumi}{(\roman{enumi})}
\item If $\mathrm{Re}(s)$ is sufficiently large, then the integral defining $Z(\wsh,s,f)$ converges absolutely for any \mbox{$\wsh \in \mathcal W^\psi_{\mathrm{Sh}}(\pi_n)$} and matrix coefficient $f$ of $\tau$.
\item $Z(\wsh,s,f)$ admits a meromorphic continuation to all of $\cc$ and there is a (unique) polynomial $P(X)\in\cc[X]$ such that $P(0)=1$ and 
$$I(\pi,\tau):=\langle Z(\wsh,s,f)\ | \ \wsh\in \mathcal W^{\psi}_{\mathrm{Sh}}(\pi_n), \ f: \text{a matrix coefficient of } \tau \rangle_\cc=P(q^{-s})^{-1}\cc[q^{-s},q^{s}];$$
denote $P(q^{-s})^{-1}$  by $L(\pi;s,\tau)$.
\item If $\tau=\tau_1\times\tau_2$ for  $\tau_1,\tau_2\in\mathrm{Alg'}$, then
$$L(\pi;s,\tau)=L(\pi;s,\tau_1)L(\pi;s,\tau_2).$$
\item  
There is a function $\gamma(s)\in\cc(q^{-s})$ such that 
$$Z(\widetilde\wsh,1-s,f^\vee)=\gamma(s)Z(\wsh,s,f)$$
for any $\wsh$ and $f$; denote $\gamma(s)$  by $\gamma(\pi;s,\tau,\psi)$.
\item If $\tau$ is a subrepresentation of $\tau_1\times\tau_2$ for $\tau_1,\tau_2\in\mathrm{Alg'}$, then
$$\gamma(\pi;s,\tau,\psi)=\gamma(\pi;s,\tau_1,\psi)\gamma(\pi;s,\tau_2,\psi).$$
\item 
If $\tau\in \mathrm{Irr_{gen}}\gll_n$, then
$$L(\pi;s,\tau)=L(s,\pi\boxtimes\tau),$$
where the right-hand side is the local L-factor defined by Jacquet, Piatetski-Shapiro and Shalika \cite{JPSS}.
\end{enumerate}
\end{theo}

\begin{rem}\label{R:sad}
As mentioned in Remark \ref{R:sadsad}, some of our results are known.
Specifically, they are as follows:
\begin{itemize}
\item (i)  follows from the discussion at the beginning of Appendix C of  \cite{cfk}.
\item (ii) is a specialized version of  \cite[Theorem C.6]{cfk}.
\item (iv) is a conclusion of  \cite[Theorem C.1]{cfk}.
\item (v) is a specialized version of  \cite[Theorem C.2]{cfk}.
\end{itemize}
On the other hand:
\begin{itemize}
\item (iii) is a partial refinement of \cite[Theorem C.6 (2)]{cfk}, which says that $$L(\pi;s,\tau)\in L(\pi;s,\tau_1)L(\pi;s,\tau_2)\cc[q^{-s},q^s].$$

\item (vi)  is a partial refinement of \cite[Proposition C.10]{cfk}, which says that $$ L(s,\pi\boxtimes\tau)\in L(\pi;s,\tau)\cc[q^{-s},q^s].$$
\end{itemize}
The proofs of (iii) and (vi) (essentially, Proposition \ref{P:gen2}, Proposition \ref{P:red1}, and Proposition \ref{P:red2}) depend on the facts in \S \ref{SS:mod}.
Only if $k=2$, then $N'_{\bm n,2}$ (in \S \ref{SS:mod}) coincides with the unipotent radical of some upper triangular parabolic subgroup of $\gll_n\subset\gll_{2n}$, so that, by partial integration, we can obtain other representations of our zeta integral using intermediate models.
Our method does not seem to be (immediately) applicable to the general case in \cite{cfk} if $k>2$.
\end{rem}

\begin{rem}\label{R:main1}
\begin{enumerate}
\renewcommand{\labelenumi}{(\roman{enumi})}
\item 
By a simple computation, one can see that the $\psi$-dependence of $\gamma(\pi;s,\tau,\psi)$ is given by
$$\gamma(\pi;s,\tau,\psi(a\cdot))=\chi_\pi^n(a)\chi_\tau^2(a)|a|^{n(2s-1)}\gamma(\pi;s,\tau,\psi)$$
for any $a\in F^\times$.
\item We put 
$$\epsilon(\pi;s,\tau,\psi):=\gamma(\pi;s,\tau,\psi)L(\pi;s,\tau)/L(\pi^\vee;1-s,\tau^\vee).$$
Then, we have $\epsilon(\pi;s,\tau,\psi), \epsilon(\pi^\vee;1-s,\tau^\vee,\psi)\in \cc[q^{-s},q^{s}]$ by Theorem \ref{T:main1}(ii), (iv) and 
$$\epsilon(\pi;s,\tau,\psi)\epsilon(\pi^\vee;1-s,\tau^\vee,\psi)=1$$
by the above (i) and using Theorem \ref{T:main1}(iv) twice.
In particular, we can write $\epsilon(\pi;s,\tau,\psi)=cq^{ls}$ for some $c\in\cc$ and $l\in\zz$.
\end{enumerate}
\end{rem}
We prove Theorem \ref{T:main1} in several parts:
The proof of (ii) is given in \S \ref{SS:red} by reduction to the generic case.
(ii) for generic representations is proved in \S \ref{SS:az} (Corollary \ref{C:genthm}).
The proofs of (iii) and (v) are given in \S \ref{SS:red} as byproducts of the proof of (ii).
The proof of (iv) is a bit technical, so we postpone it to the next section.
(vi) is proved using the entire \S \ref{SS:pf6}.

Finally, (i) is proved here.
To prove this, we need the following two lemmas.
\begin{lemm}\label{L:ks}
\begin{enumerate}
\renewcommand{\labelenumi}{(\roman{enumi})}
\item For any $\wsh\in\mathcal W^{\psi}_{\mathrm{Sh}}(\pi_n)$, there is an open compact subset $C$ of $\mathrm M_n(F)$ such that \mbox{$\mathrm{supp} \Phi_{\wsh}\subset C$}.
\item If $r\in\rr$ is sufficiently large, then $\Phi_{\wsh}|\cdot|^r$ is $L^2$-function for any $\wsh\in\mathcal W^{\psi}_{\mathrm{Sh}}(\pi_n)$.
\end{enumerate}
\end{lemm}
\begin{proof}
(i) and (ii) immediately follow from \cite[Lemma 3.2]{lapid_mao_2020} and \cite[Proposition 4.2]{lapid_mao_2020}, respectively.
\end{proof}

\begin{lemm}\label{L:L2}
Let  $G$ be an algebraic group over $F$ and $K$ an open compact subgroup of $G$.
Let $\mathcal F$ be a right $K$-invariant function on $G$.
If $\mathcal F$ is an $L^1$-function, then $\mathcal F$ is an $L^2$-function.
\end{lemm}
\begin{proof}
We have the lemma as follows:
\begin{align*}
\int_G |\mathcal F(g)|^2dg&=\int_{G\times K} |\mathcal F(g)\mathcal F(gk)|dgdk\leq\int_{G\times G} |\mathcal F(g)\mathcal F(gg')|dgdg'=\left(\int_G |\mathcal F(g)|dg\right)^2<\infty.
\end{align*}
\end{proof}
\begin{proof}[proof of Theorem \ref{T:main1} (i)]
By Lemma  \ref{L:ks} (i), we have
$$\Phi_{\wsh} (g)f(g)|\det g|^{s-\frac{1}{2}}=\Phi_{\wsh} (g)|\det g|^{s_1}\times 1_{C\cap\gll_n}(g)f(g)|\det g|^{s_2}$$
for any $\wsh$, $f$, and $g\in\gll_n$, where $C$ is some open compact subset of $\mathrm M_n(F)$ and $s-1/2=s_1+s_2$.
By Lemma \ref{L:ks} (ii), $\Phi_{\wsh}|\cdot|^{s_1}$ is an $L^2$-function for any $\wsh$ if $\mathrm {Re}(s_1)$ is sufficiently large.
On the other hand, $1_{C\cap\gll_n}f|\cdot |^{s_2}$ is an $L^2$-function for any $f$ by  Lemma \ref{L:L2} and the convergence of the zeta integral of Godement and Jacquet \cite[Theorem 3.3]{Godement_1972} if $\mathrm {Re}(s_2)$ is sufficiently large.
Thus the integral defining $Z(\wsh,s,f)$ converges absolutely for any $\wsh$ and $f$ if $\mathrm{Re}(s)$ is sufficiently large.
\end{proof}

\subsection{The generic case}\label{SS:az}
In this subsection, we assume $\tau\in \mathrm{Irr_{gen}}\gll_n$ (we do not need to assume that $\pi$ is approximately tempered here).
We let $\mathcal W^\psi(\tau')=\mathcal W^\psi_{\mathrm{Ze}}(\tau')$ for any $\tau'\in {\rm Irr_{gen}}$.

For any open compact subgroup $K$ of $\gll_n$ and $h\in\gll_n$, define $L_{h,K} \in\mathcal W^\psi(\tau)^\vee$ by
$$\langle W,L_{h,K}\rangle=\int_K W(hk) dk$$
for $W\in \mathcal W^\psi(\tau)$.
We can obtain $\mathcal W^\psi(\tau)^\vee=\langle L_{h,K}\rangle_\cc$ easily.

For any $\wsh \in \mathcal W^\psi_{\mathrm{Sh}}(\pi_n)$ and $W\in\mathcal W^\psi(\tau)$, we consider the following zeta integral
$$Z(\wsh,s,W)=\int_{\gll_n}\Phi_{\wsh} (g)W(g)|\det g|^{s-\frac{1}{2}}dg.$$

\begin{prop}\label{P:gen1}
\begin{enumerate}
\renewcommand{\labelenumi}{(\roman{enumi})}
\item If $\mathrm{Re}(s)$ is sufficiently large, then the integral defining $Z(\wsh,s,W)$ converges absolutely for any $\wsh \in \mathcal W^\psi_{\mathrm{Sh}}(\pi_n)$ and $W\in\mathcal W^\psi(\tau)$.
\item For any matrix coefficient $f$ of $\tau$ and $\wsh\in\mathcal W^{\psi}_{\mathrm{Sh}}(\pi_n)$, there are $W_{\mathrm{Sh}}^i\in \mathcal W^{\psi}_{\mathrm{Sh}}(\pi_{n})$ and $W^{i}\in \mathcal W^\psi(\tau)$ $(i=1,\dots, l)$ such that
$$Z(W_{\mathrm{Sh}},s,f)=\sum_{1\leq i\leq l}Z(W^i_{\mathrm{Sh}},s,W^i)$$
if $\mathrm{Re}(s)$ is sufficiently large.
\item For any  $W\in \mathcal W^\psi(\tau)$ and $\wsh\in\mathcal W^{\psi}_{\mathrm{Sh}}(\pi_n)$, there is a matrix coefficient $f$ of $\tau$ such that
$$Z(W_{\mathrm{Sh}},s,W)=Z(W_{\mathrm{Sh}},s,f)$$
if $\mathrm{Re}(s)$ is sufficiently large.
\end{enumerate}
\end{prop}
\begin{proof}
(i) For any sufficiently large $r$, $1_{C\cap\gll_n}W|\cdot|^r$ is an $L^1$-function of $\gll_n$ for any open compact subset $C$ of $\mathrm M_n(F)$ and $W\in \mathcal W^\psi(\tau)$ (see \cite[(2.3.6), (3.1)]{JPS1}).
Thus it immediately holds by Lemma \ref{L:ks}, \ref{L:L2}.

(ii) We can assume that $f(g)=\langle W(\cdot g), L_{h,K}\rangle$ for some $W\in\mathcal W^\psi(\tau)$ and open compact subgroup $K$ of $\gll_n$ such that $\Phi_{\wsh}$ is bi-$K$-invariant.
Then, by (i), we have
$$Z(\wsh,s,f)=\int_{\gll_n}\Phi_{\wsh} (h^{-1}gh)W(gh)|\det g|^{s-\frac{1}{2}}dg=Z(\chi(\det h)^{-1}\wsh(\cdot \diag (h,h)^{w_n}),s,W(\cdot h))$$
 if $\mathrm{Re}(s)$ is sufficiently large.

(iii) Take sufficiently small $K$ and put $f(g)=\langle W(\cdot g),L_{1,K}\rangle$.
\end{proof}

\begin{prop}\label{P:gen2}
Let $\wsh \in \mathcal W^\psi_{\mathrm{Sh}}(\pi_n)$ and $W\in\mathcal W^\psi(\tau)$.
Then, $L(s,\pi\boxtimes\tau)^{-1}Z(\wsh,s,W)$ defines an entire function of $s$ and $Z(\wsh,s,W)\in L(s,\pi\boxtimes\tau)\cc[q^{-s},q^s]$.
\end{prop}
\begin{proof}
First we remark that for any $W'\in\mathcal W^\psi(\tau') \ (\tau'\in {\rm Irr_{gen}}\gll_{m})$, there is a constant $C'>0$ such that 
$$W'\left(u\diag(a_1,\dots,a_{m'})k \right)\neq0\Rightarrow |a_1|\leq C'|a_2|\leq\dots\leq {C'}^{m-1}|a_{m}|$$
for any $u\in U_{(1^{m})},a_i\in F^\times$ and $k\in \gll_{m}(\mathcal O)$ in general (\cite[Lemma 3.2]{lapid_mao_2020}).

If $n=1$, then $Z(\wsh,s,W)$ coincides with a Rankin-Selberg zeta integral for $\pi\boxtimes \tau$.
In this case, $Z(\wsh,s,W)\in L(s,\pi\boxtimes\tau)\cc[q^{-s},q^s]$ is trivial.

Assume that $n\geq2$.
Then, for any $u\in U_{(1^n)},h\in\gll_n,$ and $h'\in\gll_{2n}$, we have
$$\wsh\left(\begin{pmatrix}u&0\\0&1_n\end{pmatrix}^{w_n} h'\right)W(uh)=\wsh\left(\begin{pmatrix}u&0\\0&1_n\end{pmatrix}^{w_n}
x_{ n}
 h'\right)W(h),$$
where we put
$$x_m:=\begin{pmatrix}
1_m&\begin{smallmatrix}
0&0\\
1_{m-1}&0
\end{smallmatrix}\\
&1_m
\end{pmatrix}^{w_m}=\diag(1,\overbrace{(\begin{smallmatrix}
1&0\\
1&1
\end{smallmatrix}),\dots,(\begin{smallmatrix}
1&0\\
1&1
\end{smallmatrix})}^{m-1},1).$$
Thus we have 
\begin{align*}
Z(\wsh,s,W)&=\int_{U_{(1^n)}\backslash\gll_n}W_{\rm Ze}\left(x_{ n}
\begin{pmatrix}
g_n&\\
&1_n
\end{pmatrix}^{w_n}
\right)W(g_n)|\det g_n|^{s-\frac{1}{2}}d\overline{g}_n
\end{align*}
if $\mathrm{Re}(s)$ is sufficiently large, where $W_{\rm Ze}:=\mathcal T_n^{-1}\wsh\in\mathcal W^\psi_{\mathrm{Ze}}(\pi_n).$
We can write
\begin{align*}
Z(\wsh,s,W)&=\int_{U_{(1^2)}\backslash\gll_2\times (F^\times)^{n-2}\times \gll_n(\mathcal O)}W_{\rm Ze}\left(
x_n\begin{pmatrix}
\diag(g_2,a_3,\dots,a_n)k&\\
&1_n
\end{pmatrix}^{w_n}
\right)\\
&W(\diag(g_2,a_2,\dots,a_n)k)|\det g_2|^{s-n+2-\frac{1}{2}}|a_3|^{s-n+5-\frac{1}{2}}\dots|a_n|^{s+n-1-\frac{1}{2}}d\overline{g}_2d^\times a_3\dots d^\times a_n dk\\
&=\int_{(F^\times)^{n}\times \gll_2(\mathcal O)\times\gll_n(\mathcal O)}W_{\rm Ze}\left(
x_n
\begin{pmatrix}
\diag((\begin{smallmatrix}
a_1&\\
&a_2
\end{smallmatrix}))k_2,a_3,\dots,a_n)k&\\
&1_n
\end{pmatrix}^{w_n}
\right)\\
&W(\diag((\begin{smallmatrix}
a_1&\\
&a_2
\end{smallmatrix}))k_2,a_2,\dots,a_n)k)|a_1|^{s-n+1-\frac{1}{2}}\dots|a_n|^{s+n-1-\frac{1}{2}}d\overline{g}_2d^\times a_3\dots d^\times a_n dk_2dk.
\end{align*}

Let $b_i\in F^\times \ (i=1,\dots,n)$ such that $|b_1|\leq C|b_2|\leq\dots\leq {C}^{n-1}|b_n|$, where $C$ is the constant in the above remark for $W$.
Consider when the inequality
\begin{align*}
\mathcal F(b_1,\dots,b_n,k')&:=W_{\rm Ze}
\left(
x_n\begin{pmatrix}
\diag(b_1,b_2,\dots,b_n)k'&\\
&1_n
\end{pmatrix}^{w_n}
\right)\\
&\left(=W_{\rm Ze}
\left(
\begin{pmatrix}
1&&&&&&&&\\
&1&&&&&&&\\
&1&1&&&&&&\\
&&&1&&&&&\\
&&&1&1&&&&\\
&&&&&\ddots&&&\\
&&&&&&1&&\\
&&&&&&1&1&\\
&&&&&&&&1
\end{pmatrix}
\begin{pmatrix}
b_1&&&&&&\\
&1&&&&&\\
&&b_2&&&&\\
&&&1&&&\\
&&&&\ddots&&\\
&&&&&b_n&\\
&&&&&&1
\end{pmatrix}
\begin{pmatrix}
k'&\\
&1_m
\end{pmatrix}^{w_n}
\right)\right)\\
&\neq 0
\end{align*}
holds for some  $k'\in \gll_n(\mathcal O)$.
If $|b_n|$ is sufficiently large, then
\begin{align*}
&\mathcal F(b_1,\dots,b_n,k')\\
=&W_{\rm Ze}\left(
\begin{pmatrix}
1&&&&&&\\
&1&&&&&\\
&1&1&&&&\\
&&&\ddots&&&\\
&&&&1&&\\
&&&&1&1&\\
&&&&&&1_3
\end{pmatrix}
\begin{pmatrix}
b_1&&&&&&\\
&1&&&&&\\
&&b_2&&&&\\
&&&1&&&\\
&&&&\ddots&&\\
&&&&&b_n&\\
&&&&&&1
\end{pmatrix}
\begin{pmatrix}
1_{2n-2}&&\\
&1&b_n^{-1}\\
&&1
\end{pmatrix}
\begin{pmatrix}
k'&\\
&1_n
\end{pmatrix}^{w_m}
\right)\\
=&W_{\rm Ze}\left(
\begin{pmatrix}
1&&&&&&\\
&1&&&&&\\
&1&1&&&&\\
&&&\ddots&&&\\
&&&&1&&\\
&&&&1&1&\\
&&&&&&1_3
\end{pmatrix}
\begin{pmatrix}
b_1&&&&&&\\
&1&&&&&\\
&&b_2&&&&\\
&&&1&&&\\
&&&&\ddots&&\\
&&&&&b_n&\\
&&&&&&1
\end{pmatrix}
\begin{pmatrix}
k'&\\
&1_n
\end{pmatrix}^{w_n}
\right)=0
\end{align*}
by Remark \ref{R:div} and the above remark.
If $|b_n|$ is sufficiently small, then, since
$$\begin{pmatrix}
1&&&&&&\\
&1&&&&&\\
&1&1&&&&\\
&&&\ddots&&&\\
&&&&1&&\\
&&&&1&1&\\
&&&&&&1
\end{pmatrix}=
\begin{pmatrix}
1&&&&&&\\
&-1&1&&&&\\
&&1&&&&\\
&&&\ddots&&&\\
&&&&-1&1&\\
&&&&&1&\\
&&&&&&1
\end{pmatrix}
w'
\begin{pmatrix}
1&&&&&&\\
&1&1&&&&\\
&&1&&&&\\
&&&\ddots&&&\\
&&&&1&1&\\
&&&&&1&\\
&&&&&&1
\end{pmatrix}
$$
($w':=\diag(1\overbrace{(\begin{smallmatrix}
&1\\
1&
\end{smallmatrix}
),\dots,(\begin{smallmatrix}
&1\\
1&
\end{smallmatrix}
)}^{n-1},1)$),
we have
\begin{align*}
&\mathcal F(b_1,\dots,b_n,k')\\
=&W_{\rm Ze}\left(
\begin{pmatrix}
b_1&&&&&&&\\
&-b_2&&&&&&\\
&&1&&&&&\\
&&&-b_3&&&&\\
&&&&\ddots&&&\\
&&&&&1&&\\
&&&&&&-b_n&\\
&&&&&&&1_2
\end{pmatrix}
w'
\begin{pmatrix}
1&&&&&&\\
&1&b_2&&&&\\
&&1&&&&\\
&&&\ddots&&&\\
&&&&1&b_n&\\
&&&&&1&\\
&&&&&&1
\end{pmatrix}
\begin{pmatrix}
k'&\\
&1_n
\end{pmatrix}^{w_n}
\right)\\=&W_{\rm Ze}\left(
\begin{pmatrix}
b_1&&&&&&&\\
&-b_2&&&&&&\\
&&1&&&&&\\
&&&-b_3&&&&\\
&&&&\ddots&&&\\
&&&&&1&&\\
&&&&&&-b_n&\\
&&&&&&&1_2
\end{pmatrix}
w'
\begin{pmatrix}
k'&\\
&1_n
\end{pmatrix}^{w_n}
\right)=0
\end{align*}
also by Remark \ref{R:div} and the above remark.
Then, by repeating a similar argument, we have that there is a constant $c>0$ such that if
$\mathcal F(b_1,\dots,b_n,k')W\left(\diag (b_1,\dots,b_n)k' \right) \neq 0$ for some $k'\in \gll_n(\mathcal O)$, then $|b_2|,|b_3|^{\pm1},|b_4|^{\pm1},\dots,|b_n|^{\pm1}<c$.

Then, integrating in $a_i \ (i>2)$ and $k$, and dividing the integral interval with respect to $a_2$, we have
\begin{align*}
Z(\wsh,s,W)=\sum_i F_i\int_{U_{(1^2)}\backslash\gll_2}W'_{\pi i}(g_2)W'_i(\diag(g_2,1_{n-2}))|\det g_2|^{s-\frac{n-2}{2}}1_{x(\mathcal O, \mathcal O)}((\begin{smallmatrix}
0&1
\end{smallmatrix})g_2
)d\overline g_2\\
+\sum_j F'_j\int_{F^\times}W_{\pi j}''(\diag(a,1))W''_j(\diag(a,1_{n-1}))|a|^{s-\frac{n}{2}}d^\times a
\end{align*}
for some $F_i,F_j'\in\cc[q^{-s},q^s], W'_{\pi i}, W''_{\pi j}\in\mathcal W_{\rm Ze}^{\psi^{-1}}(\pi),W'_{i}, W''_{j}\in\mathcal W^{\psi}(\tau)$ ($i=1,\dots,l_1,j=1,\dots,l_2$) and $x\in F^\times$ such that $|x|\ll c$ by Remark \ref{R:div}.
Since the set $\{\Phi_{W_{\pi}} \ | \  W_{\pi}\in \mathcal W^\psi(\pi) \}$ has already been calculated explicitly (see \cite[Theorem 4.7.2, 4.7.3]{Bump_1997}), it is easy to
check that the latter sum is an element of $L(s,\pi\boxtimes\tau)\cc[q^{-s},q^s]$.
If $n=2$, then the integrals in the former sum are Rankin-Selberg zeta integrals for $\pi\boxtimes \tau$.
If $n>2$, then for each $i$, we have
\begin{align*}
&\int_{U_{(1^2)}\backslash\gll_2}W'_{\pi i}(g_2)W_i'(\diag(g_2,1_{n-2}))|\det g_2|^{s-\frac{n-2}{2}}1_{x(\mathcal O, \mathcal O)}((\begin{smallmatrix}
0&1
\end{smallmatrix})g_2
)d\overline g_2\\
=&\int_{U_{(1^2)}\backslash\gll_2}W'_{\pi i}(g_2)W'_i(\diag(g_2,1_{n-2}))|\det g_2|^{s-\frac{n-2}{2}}d\overline g_2\\
-&\int_{F^\times}\int_{|x|<|b|}\int_{\gll_2(\mathcal O)}W_{\pi i}'(\diag(a,b)k_2)W_i'(\diag(\diag(a,b)k_2,1_{n-2}))|a|^{s-\frac{n}{2}}|b|^{s-\frac{n}{2}-2}dkd^\times b d^\times a.
\end{align*}
The former integral is a Rankin-Selberg zeta integral for $\pi\boxtimes \tau$.
On the other hand, if $|b|$ is sufficiently large, then $W'(\diag(\diag(a,b)k_2,1_{n-2})=0$ for any $a$ and $k_2$.
Therefore,  integrating in $b$ and $k_2$, we can see that the latter integral is an element of $L(s,\pi\boxtimes\tau)\cc[q^{-s},q^s]$.

Consequently, we have  $Z(\wsh,s,W)\in L(s,\pi\boxtimes\tau)\cc[q^{-s},q^s]$.
\end{proof}
We can see that $I(\pi,\tau)(=\langle Z(\wsh,s,f)\ | \ \wsh\in \mathcal W^{\psi}_{\mathrm{Sh}}(\pi_n), \ f: \text{a matrix coefficient of } \tau \rangle_\cc)$ is a nonzero fractional ideal of $\cc[q^{-s},q^s]$, in the same way as for the space generated by the zeta integrals of Godement and Jacquet for any admissible representation of any general linear group (see the discussion below Theorem 3.3 of \cite{Godement_1972}). 
Then, by Proposition \ref{P:gen1} and \ref{P:gen2}, we have the following:

\begin{cor}\label{C:genthm}
We have 
\begin{align*}
I(\pi,\tau)&=\langle Z(\wsh,s,W)\ | \ \wsh\in \mathcal W_{\mathrm{Sh}}(\pi_n), \ W\in\mathcal W^\psi(\tau) \rangle_\cc\\
&\subset L(s,\pi\boxtimes\tau)\cc[q^{-s},q^s]
\end{align*}
In particular, Theorem \ref{T:main1} (ii) holds for irreducible generic $\tau$ without assuming that $\pi$ is approximately tempered.
\end{cor}
For the proof of Theorem \ref{T:main1} (vi), we end this subsection with the following lemma.
\begin{lemm}\label{L:12}
If $n=1,2$ and $\tau$ is irreducible and supercuspidal, then we have
$L(\pi;s,\tau)=L(s,\pi\boxtimes\tau)$.
\end{lemm}
\begin{proof}
If $n=1$, then it is trivial.

Assume $n=2$.
Then, for any $\wsh \in \mathcal W^\psi_{\mathrm{Sh}}(\pi_2)$ and $W\in\mathcal W^\psi(\tau)$, similar to Proposition \ref{P:gen1}, we have
\begin{align*}
Z(\wsh,s,W)&=\int_{U_{(1^2)}\backslash\gll_2}W_{\rm Ze}\left(
\begin{pmatrix}
\diag(1,-1)g_2
&\\
&1_2
\end{pmatrix}{w_2}
\right)W(g_2)|\det g_2|^{s-\frac{1}{2}}1_{x(\mathcal O, \mathcal O)}((\begin{smallmatrix}
0&1
\end{smallmatrix})g_2
)d\overline{g}_2\\
+&\int_{F^\times}\int_{|x|<|b|}\int_{\gll_2(\mathcal O)}
W_{\rm Ze}\left(
x_2
\begin{pmatrix}
\diag(a,b)k&\\
&1_2
\end{pmatrix}^{w_2}
\right)
W(\diag(a,b)k_2)|a|^{s-1}|b|^{s-3}dkd^\times b d^\times a.
\end{align*}
for some $x\in F^\times$.
Then, since $\tau$ is supercuspidal, the latter integral is an  element of $\cc[q^{-s},q^{s}]$.
Thus, clearly the functions $Z(\wsh,s,W)$ generate $L(s,\pi\boxtimes \tau)\cc[q^{-s},q^s]$.
\end{proof}

\subsection{More results in the case that $\pi$ is approximately tempered}\label{SS:red}
For the rest of this section, we assume that $\pi$ is approximately tempered.
Then, for any $W\in \mathcal W^\psi_{(n_1,n_2)}(\pi_n)$,
\begin{align}W|_{\diag( \gll_{2n_1},\gll_{2n_2})}\in \mathcal \wsh^\psi(\pi_{n_1})|\cdot|^{\frac{1}{2}n_2}\otimes\mathcal \wsh^\psi(\pi_{n_2})|\cdot|^{-\frac{1}{2}n_1}\end{align}
holds, where $n_1+n_2=n$ (see Remark \ref{R:div}).
We fix  $\tau_i\in\mathrm{Alg'}\gll_{n_i}$ ($i=1,2$) and put $\tau_0=\tau_1\times\tau_2$.

Similar to $L_{h,K}$ in the previous subsection, for any open compact subgroup $K$ of $\gll_n$, $h\in\gll_n$, and $v_i^\vee\in\tau_i$ ($i=1,2$), we define $L_{h,K,v_1,v_2}\in\tau_0^\vee\simeq\tau_1^\vee\times\tau_2^\vee$ by 
$$\langle v,L_{h,K,v_1^\vee,v_2^\vee}\rangle=\int_K\langle v(hk),v_1^\vee\otimes v_2^\vee\rangle dk$$
for any $v\in\tau_1\times\tau_2$.
Then, it is easy to verify that $\tau_0^\vee=\langle L_{h,K,v_1^\vee,v_2^\vee}\rangle_\cc.$
\begin{prop}\label{P:red1}
Let $\wsh\in\mathcal W^{\psi}_{\mathrm{Sh}}(\pi_n)$ and $f$ a matrix coefficient of $\tau_0$.
\begin{enumerate}
\renewcommand{\labelenumi}{(\roman{enumi})}
\item There are $W_{\mathrm{Sh},i}^j\in \mathcal W^{\psi}_{\mathrm{Sh}}(\pi_{n_j})$ and matrix coefficients $f_{i}^j$ of $\tau_i$ ($i=1,\dots,l,  \ j=1,2$) such that
$$Z(\wsh,s,f)=\sum_{1\leq i \leq l}Z(W_{\mathrm{Sh},i}^1,s,f_{i}^1)Z(W_{\mathrm{Sh},i}^2,s,f_{i}^2)$$
if $\mathrm{Re}(s)$ is sufficiently large.
\item
For the above  $W_{\mathrm{Sh},i}^j$ and $f_{i}^j$, we have
$$Z(\widetilde\wsh,1-s,f^\vee)=\sum_{1\leq i \leq l}Z(\widetilde {W_{\mathrm{Sh},i}^1},1-s,(f_{i}^1)^\vee)Z(\widetilde {W_{\mathrm{Sh},i}^2},1-s,(f_{i}^2)^\vee)$$
 if $\mathrm{Re}(s)$ is sufficiently small.
\end{enumerate}
\end{prop}

\begin{proof}
(i) We can assume that $f(g)=\langle v(\cdot g), L_{h,K,v_1^\vee,v_2^\vee}\rangle$, where $K$ is an open compact subgroup of $\gll_n$ such that $\Phi_{\wsh}$ is bi-$K$-invariant and $v\in \tau_0$.
Since $1_{C\cap\gll_n}\langle v(\cdot),v_1^\vee\otimes v_2^\vee\rangle|\cdot|^r$ is an $L^1$-function of $\gll_n$ for any compact subset $C$ of $\mathrm M_n(F)$ and sufficiently large $r$ by Iwasawa decomposition and \cite[Theorem 3.3]{Godement_1972}, if $\mathrm{Re}(s)$ is sufficiently large, then we have
\begin{align*}
Z(\wsh,s,f)&=\int_{\gll_n\times K}\Phi_{\wsh} (g)\langle v(hkg),v_1^\vee\otimes v_2^\vee\rangle|\det g|^{s-\frac{1}{2}}dgdk\\
&=\int_{\gll_n}\Phi_{\wsh} (h^{-1}gh)\langle v(gh),v_1^\vee\otimes v_2^\vee\rangle|\det g|^{s-\frac{1}{2}}dg
\end{align*}
by Lemma \ref{L:ks}, \ref{L:L2}.
By Iwasawa decomposition $\gll_n= U_{(n_1,n_2)}M_{(n_1,n_2)}\gll_n(\mathcal O)$, the above integral can be written as
\begin{align*}
\int_{\gll_{n_1}\times\gll_{n_2} \times\gll_n(\mathcal O)}W
 \left(
\begin{pmatrix}
(\begin{smallmatrix}
g_1&\\
&1_{n_1}
\end{smallmatrix})^{w_{n_1}}
&\\
&(\begin{smallmatrix}
g_2&\\
&1_{n_2}
\end{smallmatrix})^{w_{n_2}}
\end{pmatrix}
\begin{pmatrix}
kh&\\
&h
\end{pmatrix}^{w_{n}}
\right)
\delta_{P_{(n_1,n_2)}}(
(\begin{smallmatrix}
g_1&\\
&g_2
\end{smallmatrix})
)^{-\frac{1}{2}}\\
\langle \tau_1(g_1)\otimes\tau_2(g_2)v(kh),v_1^\vee\otimes v_2^\vee\rangle|\det g|^{s-\frac{1}{2}}dg_1dg_2dk
\times\chi_\pi(\det h)^{-1},
\end{align*}
where $$W=\int_{\mathrm M_{n_1,n_2}(F)}\wsh\left(
\begin{pmatrix}
1_{n_1}&X&\\
&1_{n_2}&\\
&&1_n
\end{pmatrix}^{w_n}
\cdot\right)dX \in \mathcal W^{\psi}_{(n_1,n_2)}(\pi_{n})$$
(Corollary \ref{C:im}).
Since the integration over $\gll_n(\mathcal O)$ becomes a finite sum, we have (i) by (1).

(ii) We have
\begin{align*}
Z(\widetilde\wsh,1-s,f)&=\int_{\gll_n\times K}\Phi_{\widetilde\wsh} (g)\langle v(hkg^{-1}),v_1^\vee\otimes v_2^\vee\rangle|\det g|^{-s+\frac{1}{2}}dgdk\\
&=\int_{\gll_n}\Phi_{\widetilde\wsh} (h^{-1}gh)\langle v(g^{-1}h),v_1^\vee\otimes v_2^\vee\rangle|\det g|^{-s+\frac{1}{2}}dg.
\end{align*}
if $\mathrm{Re}(s)$ is sufficiently small.
By Iwasawa decomposition $\gll_n= \gll_n(\mathcal O) U_{(n_1,n_2)}M_{(n_1,n_2)}$, the above integral can be written as
\begin{align*}
&\int_{\gll_{n_1}\times\gll_{n_2} \times\gll_n(\mathcal O)}W
 \left(
\begin{pmatrix}
(\begin{smallmatrix}
g_1&\\
&1_{n_1}
\end{smallmatrix})^{w_{n_1}}
&\\
&(\begin{smallmatrix}
g_2&\\
&1_{n_2}
\end{smallmatrix})^{w_{n_2}}
\end{pmatrix}
\left(
\begin{pmatrix}
h&\\
&k^{-1}h
\end{pmatrix}
\begin{pmatrix}
&1_n\\
1_n&
\end{pmatrix}
\right)^{w_{n}}
\right)
\delta_{P_{(n_1,n_2)}}((\begin{smallmatrix}
g_1&\\
&g_2
\end{smallmatrix})
))^{-\frac{1}{2}}\\
&\chi_\pi(\det g_1)^{-1}\chi_\pi(\det g_2)^{-1}\langle \tau_1(g^{-1}_1)\otimes\tau_2(g^{-1}_2)v(k^{-1}h),v_1^\vee\otimes v_2^\vee\rangle|\det g_1\det g_2|^{-s+\frac{1}{2}}dg_1dg_2dk
\times\chi_\pi(\det h)^{-1}\\
=&\int_{\gll_{n_1}\times\gll_{n_2} \times\gll_n(\mathcal O)}W
 \left(
\begin{pmatrix}
((\begin{smallmatrix}
g_1&\\
&1_{n_1}
\end{smallmatrix})
(\begin{smallmatrix}
&1_{n_1}\\
1_{n_1}&
\end{smallmatrix})
)^{w_{n_1}}
&\\
&((\begin{smallmatrix}
g_2&\\
&1_{n_2}
\end{smallmatrix})
(\begin{smallmatrix}
&1_{n_2}\\
1_{n_2}&
\end{smallmatrix})
)^{w_{n_2}}
\end{pmatrix}
\begin{pmatrix}
kh&\\
&h
\end{pmatrix}^{w_{n}}
\right)
\delta_{P_{(n_1,n_2)}}((\begin{smallmatrix}
g_1&\\
&g_2
\end{smallmatrix})
))^{-\frac{1}{2}}\\
&\chi_\pi(\det g_1)^{-1}\chi_\pi(\det g_2)^{-1}\langle \tau_1(g^{-1}_1)\otimes\tau_2(g^{-1}_2)v(kh),v_1^\vee\otimes v_2^\vee\rangle|\det g_1\det g_2|^{-s+\frac{1}{2}}dg_1dg_2dk
\times\chi_\pi(\det h)^{-1}.
\end{align*}
Comparing the above integral and the integral in the proof of (i), we have (ii).
\end{proof}
\begin{prop}\label{P:red2}
Let $\wsh^j\in\mathcal W^{\psi}_{\mathrm{Sh}}(\pi_{n_j})$ and $f^j$ a matrix coefficient of $\tau_j$ ($j=1,2$).
Then, there are $W_{\mathrm{Sh}}\in \mathcal W^{\psi}_{\mathrm{Sh}}(\pi_{n})$ and a matrix coefficient $f$ of $\tau_0$  such that
$$Z(W_{\mathrm{Sh}}^1,s,f^1)Z(W_{\mathrm{Sh}}^2,s,f^2)=Z(W_{\mathrm{Sh}},s,f)$$
if $\mathrm{Re}(s)$ is sufficiently large.
\end{prop}
\begin{proof}
Since $\pi_m$ is irreducible for any $m$, we can take $\wsh\in \mathcal W^{\psi}_{\mathrm{Sh}}(\pi_{n})$ such that
$$\left(\int_{\mathrm M_{n_1,n_2}(F)}\wsh\left(
\begin{pmatrix}
1_{n_1}&X&\\
&1_{n_2}&\\
&&1_n
\end{pmatrix}^{w_n}
\cdot\right)dX\middle )\right|_{\diag(\gll_{2n_1},\gll_{2n_2})}
=\wsh^1|\cdot|^{\frac{n_2}{2}}\otimes\wsh^2|\cdot|^{-\frac{n_1}{2}}$$
 by Corollary \ref{C:im} and (1).
 Write $f^i=\langle \tau_i(\cdot)v_i,v_i^\vee \rangle$ ($v_i\in\tau_i, v_i^\vee\in\tau_i^\vee$) and define $v\in\tau_0$ by
 $$v(g)=\int_{P_{(n_1,n_2)}}1_K(pg)\delta_{P_{(n_1,n_2)}}(\diag (m_1,m_2))^{-\frac{1}{2}}\tau_1(m_1^{-1})v_1\otimes\tau_2(m_2^{-1})v_2d_rp,$$
 where $p=\begin{pmatrix}
m_1&*\\
&m_2
\end{pmatrix}$, $K$ is an open compact subgroup of $\gll_n$ such that $\Phi_{\wsh}$ is bi-$K$-invariant, and $d_rp$ is the right Haar measure of $P_{(n_1,n_2)}$.
Then, for $f=\langle \tau_0(\cdot)v,L_{h,K,v_1^\vee,v_2^\vee} \rangle$, we have
\begin{align*}
Z(\wsh,s,f)&=\int_{\gll_n}\Phi_{\wsh} (g)\langle v(g),v_1^\vee\otimes v_2^\vee\rangle|\det g|^{s-\frac{1}{2}}dg\\
&=\int_{\gll_n\times P_{(n_1,n_2)}}\Phi_{\wsh} (p^{-1}g)1_K(g)\delta_{P_{(n_1,n_2)}}(\diag (m_1,m_2))^{-\frac{1}{2}}f^1(m_1^{-1})f^2(m_2^{-1})|\det p^{-1}g|^{s-\frac{1}{2}}d_rpdg\\
&=\int_{\gll_{n_1}\times\gll_{n_2}}\Phi_{\wsh^1} (m_1)\Phi_{\wsh^2} (m_2)f^1(m_1)f^2(m_2)|\det m_1 \det m_2|^{s-\frac{1}{2}}dm_1dm_2\\
&=Z(W_{\mathrm{Sh}}^1,s,f^1)Z(W_{\mathrm{Sh}}^2,s,f^2)
\end{align*}
if $\mathrm{Re}(s)$ is sufficiently large.
\end{proof}
We give the proofs of Theorem \ref{T:main1} (ii), (iii), and (v):
\begin{proof}[proof of Theorem \ref{T:main1} (ii), (iii), (v)]
(ii) Embedding $\tau$ to the parabolic induction  of some irreducible generic representation and using Proposition \ref{P:red1} (i) repeatedly, we can assume that $\tau$ is generic and irreducible.
However, we have already proved (ii) for generic representations (Corollary \ref{C:genthm}).

(iii) It follows from Proposition \ref{P:red1} (i) and Proposition \ref{P:red2}.

(v) It follows from Proposition \ref{P:red1} (ii).
\end{proof}

\subsection{The proof of Theorem \ref{T:main1} (vi)}\label{SS:pf6}
We have already shown Theorem \ref{T:main1} (i), (ii), (iii), and (v) ((iv) will be shown in the next section).
At the end of this section, we prove  (vi).
We suppose that  Theorem \ref{T:main1} (iv) holds here.

\begin{proof}[proof of Theorem \ref{T:main1} (vi)]
We can assume that $\tau={\rm St}(\rho,m)$ for some $\rho\in{\rm Irr_{sc}}$ and $m$ by Theorem \ref{T:main1} (iii).
If $L(s, \pi\boxtimes\tau)=1$, then $L(\pi,s,\tau)=L(s, \pi\boxtimes\tau)=1$  by Corollary \ref{C:genthm}.
Thus we only have to consider the case where
\begin{enumerate}
\item $\rho\in \mathrm{Irr}\gll_1$, $\pi=\rho'\times\rho^{-1}|\cdot|^t$ or $\mathrm{St}(\rho^{-1}|\cdot|^t,2)$ ($\rho'\in \mathrm{Irr}\gll_1, t\in\cc \text{ s.t. } \rho'\rho|\cdot|^{-t}\neq|\cdot|^{\pm1}$)  or
\item $\rho\in\mathrm{Irr_{sc}}\gll_2$, $\pi=\rho^\vee|\cdot|^t$ ($t\in\cc$).
\end{enumerate}

By Corollary \ref{C:genthm}, we have 
$$L(\pi;s,\tau)=Q(q^{-s},q^s)L(s,\pi\boxtimes\tau)$$
and
$$L(\pi^\vee;1-s,\tau^\vee)=\tilde Q(q^{-s},q^s)L(1-s,\pi^\vee\boxtimes\tau^\vee)$$
for some $Q(X,Y),\tilde Q(X,Y)\in\cc[X,Y]$.
Dividing the second equation by the first equation, we have that
$\gamma(\pi;s,\tau)$ coincides with $Q(q^{-s},q^s)\tilde Q(q^{-s},q^s)^{-1}\gamma(s,\pi\boxtimes\tau)$ up to a unit.
By Theorem \ref{T:main1} (v), we have
$$\gamma(\pi;s,\tau)=\prod^m_{i=1}\gamma(\pi;s+m/2+1/2-i,\rho).$$
By Lemma \ref{L:12}, the equation $L(\pi;s,\rho)=L(s, \pi\boxtimes\rho)$ holds.
Thus $\gamma(\pi;s,\tau)$ coincides with $\gamma(s,\pi\boxtimes\tau)$ up to a unit.
Consequently, we have
$$Q(q^{-s},q^s)=cq^{ls}\tilde Q(q^{-s},q^s)$$
for some $c\in F^\times$ and $l\in\zz$.
Namely, $Q(q^{-s},q^s)\in\cc[q^{-s},q^s]$ is a common factor of $L(s,\pi\boxtimes\tau)^{-1}$ and $L(1-s,\pi^\vee\boxtimes\tau^\vee)^{-1}$.

Assume that $\pi$ and $\tau$ satisfy the conditions in 2.
Then, we have
$$L(s,\pi\boxtimes\tau)^{-1}=1-q^{-s-t-\frac{m-1}{2}}$$ and
$$L(1-s,\pi^\vee\boxtimes\tau^\vee)^{-1}=1-q^{-1+s+t-\frac{m-1}{2}}=-q^{-1+s+t-\frac{m-1}{2}}(1-q^{-s-t+\frac{m+1}{2}})$$
by Theorem of \cite[\S (8.2)]{JPSS}.
Thus $L(s,\pi\boxtimes\tau)^{-1}$ and $L(1-s,\pi^\vee\boxtimes\tau^\vee)^{-1}$ are relatively prime, and we have $L(\pi;s,\tau)=L(s,\pi\boxtimes\tau)$.

Assume that $\pi$ and $\tau$ satisfy the conditions in 1.
We also assume $n>1$ since we have already shown that $L(\pi;s,\tau)=L(s,\pi\boxtimes\tau)$ for $n=1$ (Lemma \ref{L:12}).
Then, we have
$$L(s,\pi\boxtimes\tau)^{-1}=
\begin{cases}
(1-q^{-s-t-\frac{n}{2}})(1-q^{-s-t-\frac{n}{2}+1}) &\text{ if } \pi=\mathrm{St}(\rho^{-1}|\cdot|^t,2); \\
(1-q^{-s-t'-\frac{n-1}{2}})(1-q^{-s-t-\frac{n-1}{2}})  &\text{ if } \pi=\rho^{-1}|\cdot|^{t'}\times\rho^{-1}|\cdot|^{t} \ (|\cdot|^{t-t'}\neq|\cdot|^{\pm 1});\\
(1-q^{-s-t-\frac{n-1}{2}})  &\text{ otherwise }
\end{cases}$$ and
$$L(1-s,\pi^\vee\boxtimes\tau^\vee)^{-1}=
\begin{cases}
(1-q^{-s-t+\frac{n}{2}+1})(1-q^{-s-t+\frac{n}{2}}) &\text{ if } \pi=\mathrm{St}(\rho^{-1}|\cdot|^t,2); \\
(1-q^{-s-t'+\frac{n+1}{2}})(1-q^{-s-t+\frac{n+1}{2}})  &\text{ if } \pi=\rho^{-1}|\cdot|^{t'}\times\rho^{-1}|\cdot|^{t} \ (|\cdot|^{t-t'}\neq|\cdot|^{\pm 1});\\
(1-q^{-s-t+\frac{n+1}{2}})  &\text{ otherwise}
\end{cases}$$
up to a unit.
Therefore, we have $L(\pi;s,\tau)=L(s,\pi\boxtimes\tau)$ unless $\pi=\rho^{-1}|\cdot|^{t\pm n}\times\rho^{-1}|\cdot|^{t}$ for some $t\in\cc$.
However, the case $\pi=\rho^{-1}|\cdot|^{t\pm n}\times\rho^{-1}|\cdot|^{t}$ does not occur since $\pi$ is approximately tempered.
\end{proof}

\begin{rem}
Using \cite[Proposition C.10]{cfk}, Theorem \ref{T:main1} (vi) follows from Theorem \ref{T:main1} (iii) and Corollary \ref{C:genthm} immediately.
\end{rem}
\section{The functional equation}\label{S:fe}
In this section, we prove the functional equation (Theorem \ref{T:main1} (iv)).
As mentioned in Remark \ref{R:sad},  it has already been proven.
However, for the convenience of the readers, we prove it without the results in \cite{cfk}.

The key to the proof is the following proposition.
\begin{prop}\label{P:uniq}
Let $\tau\in\mathrm{Irr_{sc}}\gll_n$ and $\tau'\in\mathrm{Irr}\gll_n$.
Then, we have
$$\dim_\cc{\rm Hom}_{\gll_n\times\gll_{n}}(\pi_{n},\tau\boxtimes\tau')\leq 1$$
if $\tau'=\tau^\vee\chi_\pi$ and
$$\dim_\cc{\rm Hom}_{\gll_n\times\gll_{n}}(\pi_{n},\tau\boxtimes\tau')=0$$
otherwise.
(Here, we think of $\gll_n\times\gll_n$ as the diagonal subgroup $\diag(\gll_n,\gll_n)$ of $\gll_{2n}$.)
\end{prop}
We give the proof in \S \ref{SS:pf}.

By this proposition, we obtain Theorem \ref{T:main1} (iv) as follows:
\begin{proof}[proof of Theorem \ref{T:main1} (iv)]
By Proposition \ref{P:red1} (ii), we can assume $\tau\in \mathrm{Irr_{sc}}\gll_n$.

For any $s\in\cc$, the map
$$(\wsh,v,v^\vee)\mapsto Z(\wsh,s,\langle \tau(\cdot)v,v^\vee\rangle)/L(\pi;s,\tau), \ \wsh\in \mathcal W^\psi_{\mathrm{Sh}}(\pi_n), \ v\in\tau,\  v^\vee\in\tau^\vee$$
is well-defined by Corollary \ref{C:genthm} and its linear extension on $\mathcal W^\psi_{\mathrm{Sh}}(\pi_n)\otimes\tau\otimes\tau^\vee$  defines a nonzero element of 
\begin{align*}
\Ho_{\diag(\gll_n,\gll_n)^{w_n}}(\pi_n\otimes(\tau|\cdot|^{s-\frac{1}{2}}\boxtimes\tau^\vee \chi_\pi^{-1} |\cdot|^{-s+\frac{1}{2}}),\cc)
&\simeq\Ho_{\diag(\gll_n,\gll_n)^{w_n}}(\pi_n,\tau^\vee|\cdot|^{-s+\frac{1}{2}}\boxtimes\tau\chi_\pi|\cdot|^{s-\frac{1}{2}})\\
&\simeq\Ho_{\gll_n\times\gll_n}(\pi_n,\tau^\vee|\cdot|^{-s+\frac{1}{2}}\boxtimes\tau\chi_\pi|\cdot|^{s-\frac{1}{2}}).
\end{align*}
On the other hand, it is easy to check that the linear extension of 
$$(\wsh,v,v^\vee)\mapsto Z(\widetilde\wsh,1-s,\langle v,\tau^\vee(\cdot)v^\vee\rangle)/L(\pi^\vee;1-s,\tau^\vee), \ \wsh\in \mathcal W^\psi_{\mathrm{Sh}}(\pi_n),\  v\in\tau,\  v^\vee\in\tau^\vee$$
 on $\mathcal W^\psi_{\mathrm{Sh}}(\pi_n)\otimes\tau\otimes\tau^\vee$  defines a nonzero element of the same space.
 By Proposition \ref{P:uniq}, these two maps coincide up to a constant.
 Thus, there is a function $\epsilon$ on $\cc$ such that
 $$Z(\widetilde\wsh,1-s,\langle v,\tau^\vee(\cdot)v^\vee\rangle)/L(\pi^\vee;1-s,\tau^\vee)=\epsilon(s)Z(\wsh,s,\langle \tau(\cdot)v,v^\vee\rangle)/L(\pi;s,\tau)$$
 for any $\wsh\in \mathcal W^\psi_{\mathrm{Sh}}(\pi_n),\  v\in\tau,\ v^\vee\in\tau^\vee$.
 By Theorem \ref{T:main1}(ii), we obtain 
 $$\gamma(s)=L(\pi^\vee;1-s,\tau^\vee)\epsilon(s)/L(\pi;s,\tau)\in\cc(q^{-s})$$
 as required.
\end{proof}

We introduce some additional notations.

Let $k,l$ be positive integers such that $k\geq l$.
Then, we often identify $\gll_{k-l}$ with $\diag(\gll_{k-l},1_{l})\subset \gll_k$.
Define
$$D_{l}^{k}=
\begin{pmatrix}
\gll_{k-l}&*\\
&1_l
\end{pmatrix}=
\gll_{k-l}
U_{(k-l,l)}\subset \gll_k$$
and 
$$D^{k}_{l}\rtimes D^l_m=
\begin{pmatrix}
\gll_{k-l}&*&*\\
&\gll_{l-m}&*\\
&&1_{m}
\end{pmatrix}=
D^{k}_{l}
\begin{pmatrix}
1_{k-l}&\\
&D^l_m
\end{pmatrix}
\subset D_{m}^{k}$$
 for $m\in\zz_{>0}$ such that $l\geq m$.
Moreover, for representations $\mu$ and $\sigma$ of $\gll_{k-l}$ and $D_{m}^l$, respectively, we denote 
$$\ind^{D_{m}^{k}}_{D^{k}_{l}\rtimes D_{m}^{l}}\mu\boxtimes \sigma$$
 by $\mu\rtimes\sigma$, where $\mu\boxtimes \sigma$ is regarded as a representation of $D^{k}_{l}\rtimes D_{m}^{l}$ by
$$U_{(k-l,l)}\backslash D^{k}_{l}\rtimes D_{m}^{l}\simeq \gll_{k-l}\times D_{m}^{l}.$$
We note that $(\mu\rtimes\sigma)|_{\gll_{k-m}}=\mu*\sigma|_{\gll_{l-m}}.$

Let $i$ be a positive integer such that $i\leq n$.
We define
$$N_i=\left\{
\begin{pmatrix}
X&Y\\
&1_i
\end{pmatrix}\middle | \ X\in U_{(1^i)}, Y\in\mathrm M_i(F): \text{upper triangular}
\right\}
\subset D_i^{2i}$$
and a character $\Psi_{N_i}$ of $N_i$ by
$$\Psi_{N_i}\left(
\begin{pmatrix}
X&Y\\
&1_i
\end{pmatrix}
\right)=\psi (\mathrm{tr}Y)$$
and put
$$I_i=\cind_{N_i}^{D_i^{2i}}\Psi_{N_i}
.$$

\subsection{The Kirillov-Shalika model}\label{SS:fl}
We start with the following lemma:
\begin{lemm}
\begin{enumerate}%
We regard $\psi\circ{\rm tr}$ as a character of $U_{(n,n)}$ by $U_{(n,n)}\simeq\mathrm M_n(F)$.
\renewcommand{\labelenumi}{(\roman{enumi})}
\item Any nonzero subrepresentation of $\ind^{D_n^{2n}}_{U_{(n,n)}}\psi\circ{\rm tr}$ contains \mbox{$\cind^{D_n^{2n}}_{U_{(n,n)}}\psi\circ{\rm tr}$}.
In particular, $\cind^{D_n^{2n}}_{U_{(n,n)}}\psi\circ{\rm tr}$ is the unique irreducible subrepresentation of $\ind^{D_n^{2n}}_{U_{(n,n)}}\psi\circ{\rm tr}$.
\item We have $I_n\simeq \cind^{D_n^{2n}}_{U_{(n,n)}}\psi\circ{\rm tr}$.
\end{enumerate}
\end{lemm}
(i) and (ii) are special cases of more general statements \cite[Lemma 3.12]{lapid_mao_2020} and \cite[Lemma 3.14]{lapid_mao_2020}, respectively.

Let $\chi$ be a character of $F^\times$.
Then, we can extend the action of $D_n^{2n}$ on $\ind^{D_n^{2n}}_{U_{(n,n)}}\psi\circ{\rm tr}$ to $P_{(n,n)}$ by 
$$(\diag(1_n,g))f:=\chi(\det g)f\left(\diag(g^{-1},1_n)\cdot\right)$$
 for any $g\in\gll_n$.
We denote the extended representation by $\tilde I_n^{\chi}$ and the $P_{(n,n)}$-submodule $\cind^{D_n^{2n}}_{U_{(n,n)}}\psi\circ{\rm tr}$ of $\tilde I_n^{\chi}$ by $I_n^\chi$.

Let $\mathcal W'^\psi_{\mathrm{Sh}}(\pi_n)$ be the `ordinal' Shalika model of $\pi_n$ i.e., a (unique) subspace of 
$\ind^{\gll_{2n}}_{U_{(n,n)}}\psi\circ{\rm tr}$ which realizes $\pi_n$.
We note that the isomorphism from $\mathcal W^\psi_{\mathrm{Sh}}(\pi_n)$ to $\mathcal W'^\psi_{\mathrm{Sh}}(\pi_n)$ is given by $\wsh\mapsto \wsh(w_n\cdot)$ for $\wsh\in \mathcal W^\psi_{\mathrm{Sh}}(\pi_n)$.

We define the Kirillov-Shalika model $\mathcal K_\psi(\pi_n)$ of $\pi_n$ by
$$\mathcal K_\psi(\pi_n)=\{\wsh'|_{D_n^{2n}} \ | \ \wsh'\in \mathcal W^\psi_{\mathrm{Sh}}(\pi_n)\}.$$
Then, the kernel of $\mathcal W'^\psi_{\mathrm{Sh}}(\pi_n)\surj  \mathcal K_\psi(\pi_n)$ is a $P_{(n,n)}$-module.
Therefore, as a space with ${P_{(n,n)}}$-action induced by $\pi_n|_{P_{(n,n)}}$, $\mathcal K_\psi(\pi_n)$ is a subrepresentation of $\tilde I_n^{\chi_{\pi}}$.

If $\pi$ is approximately tempered, then $\mathcal W'^\psi_{\mathrm{Sh}}(\pi_n)\surj  \mathcal K_\psi(\pi_n)$ is bijective (\cite[Corollary 4.4]{lapid_mao_2020}).
Thus, we have the following proposition.
\begin{prop}\label{P:sh}
Assume $\pi$ is approximately tempered.
Then, $\pi_n$ has a unique irreducible $D_n^{2n}$-submodule.
Moreover, this submodule is a $P_{(n,n)}$-submodule and is  isomorphic to $I_n^{\chi_{\pi}}$ as a $P_{(n,n)}$-representation.
\end{prop}
\subsection{Special cases}\label{SS:sc}
We consider  Proposition \ref{P:uniq} for some special cases.

\begin{lemm}\label{L:nsc}
Assume $\pi$ is not supercuspidal.
Then, Proposition \ref{P:uniq} holds.
\end{lemm}
\begin{proof}
Since $\pi$ is not supercuspidal, $\pi_n$ is a quotient of the degenerate principal series $\mathbf 1_{\gll_n}\chi_1\times\mathbf 1_{\gll_n}\chi_2$ for some $\chi_1,\chi_2\in\mathrm{Irr}\gll_1$. 
Then we have
$$\dim_\cc{\rm Hom}_{\gll_n\times\gll_{n}}(\pi_{n},\tau\boxtimes\tau')\leq\dim_\cc{\rm Hom}_{\gll_n\times\gll_{n}}(\mathbf 1_{\gll_n}\chi_1\times\mathbf 1_{\gll_n}\chi_2,\tau\boxtimes\tau').$$
Similar to \cite[Theorem 4.3]{HKS}, the equality
$$\dim_\cc{\rm Hom}_{\gll_n\times\gll_{n}}(\mathbf 1_{\gll_n}\chi_1\times\mathbf 1_{\gll_n}\chi_2,\tau\boxtimes\tau')=\begin{cases}
1 \ \text{ if } \tau'=\tau^\vee\chi_\pi;\\
0 \ \text{ otherwise}
\end{cases}$$
is easy to verify by the filtration of $\mathbf 1_{\gll_n}\chi_1\times\mathbf 1_{\gll_n}\chi_2$ in \cite[Lemma 2.5]{ATOBE2018281}.
\end{proof}

\begin{lemm}\label{L:n12}
Assume $\pi$ is supercuspidal.
Then, Proposition \ref{P:uniq} holds if $n=1,2$.
\end{lemm}
\begin{proof}
If $n=1$, then it follows from $\{\Phi_{W} \ | \  W\in \mathcal W^\psi(\pi) \}=\mathcal S(F^\times)$.

Assume $n=2$.
 Then 
 $$\dim_\cc{\rm Hom}_{\gll_2\times\gll_{2}}(\pi_{2},\tau\boxtimes\tau')=0$$
 if $\tau'\neq \tau^\vee\chi_\pi$
and
$$\dim_\cc{\rm Hom}_{\gll_2\times\gll_{2}}(\pi_{2},\tau\boxtimes\tau^\vee\chi_\pi)\leq 1$$
unless $\tau=\pi|\cdot|^{\frac{1}{2}}$
by \cite[Proposition 7.1]{lapid_mao_2020}.
Since  ${\rm Hom}_{\gll_n\times\gll_{n}}(\pi_{n},\tau\boxtimes\tau')\simeq{\rm Hom}_{\gll_n\times\gll_{n}}(\pi_{n},\tau'\boxtimes\tau)$ in general, we have $\dim_\cc{\rm Hom}_{\gll_2\times\gll_{2}}(\pi_{2},\tau\boxtimes\tau^\vee\chi_\pi)\leq 1$ in general.
\end{proof}

\subsection{The filtration}\label{SS:fl}
In this subsection, we give a good filtration of $\pi_n$.
Assume that $\pi$ is supercuspidal here.

The following fact is essential:
\begin{prop}\label{P:zell}
For any $\pi'\in\mathrm{Irr_{sc}}\gll_2$, we have
$$\mathrm{Sp}(\pi',n)|_{D^{2n}_1}\simeq \mathrm{Sp}(\pi',n-1)|\cdot|^{1/2}\rtimes I_1.$$
\end{prop}
This is a special case of a more general result \cite[Theorem 3.5]{zel}.

Using this, we obtain the following:
\begin{lemm}\label{L:tes}
Let $i,k$ be positive integers such that $k<2i<2n$ and $\sigma=(\sigma,V)$ be a representation of $D^{2i}_k$.
Then, the following $D_{k}^{2n}$-module
$$\Sigma=\pi_{n-i}|\cdot|^{\frac{i}{2}}\rtimes \sigma$$
has a $D_{k+1}^{2n}$-submodule $J$ such that 
$$ \Sigma|_{D_{k+1}^{2n}}/ J\simeq \pi_{n-i}|\cdot|^{\frac{i}{2}}\rtimes \sigma|_{ D^{2i}_{k+1}},\
J \simeq \pi_{n-i-1}|\cdot|^{\frac{i+1}{2}}\rtimes \cind^{D_{k+1}^{2i+2}}_G{\bm \sigma},
$$
 where 
 $$G=\begin{pmatrix}
1&*&*&*\\
&\gll_{2i-k}&&*\\
&&1&\\
&&&1_{k}
\end{pmatrix}
$$
 and $\bm \sigma=(\bm \sigma, V)$ is defined by  
 $$\bm \sigma\left(\begin{pmatrix}
1&*&a&*\\
&b&&c\\
&&1&\\
&&&1_{k}
\end{pmatrix}\right)=\psi(a)\sigma\left(\begin{pmatrix}
b&c\\
&1_{k}
\end{pmatrix}\right).
$$
\end{lemm}
\begin{proof}
Put $H=D^{2n}_{2i}\rtimes D^{2i}_{k}$.
Since
$$H\backslash D^{2n}_k/D^{2n}_{k+1}\simeq P_{(2n-2i,2i-k)}\backslash \gll_{2n-k}/D^{2n-k}_1,$$
we have
$$D^{2n}_i=HD^{2n}_{k+1}\sqcup HwD^{2n}_{k+1},$$
where $w=\begin{pmatrix}
1_{2n-2i-1}&&&\\
&&1&\\
&1_{2i-k}&&\\
&&&1_k
\end{pmatrix}$.
We note that $H\cap D_{k+1}^{2n}=D^{2n}_{2i}\rtimes D^{2i}_{k+1}$ and
\begin{align*}
w^{-1}Hw\cap D_{k+1}^{2n}&=
\begin{pmatrix}
\gll_{2n-2i-1}&*&*&*\\
&\gll_{2i-k}&&*\\
&&1&\\
&&&1_{k}
\end{pmatrix}
.
\end{align*}
Put $ J=\{f \in \Sigma \ | {\rm supp} f \subset HwD^{2n}_{i+1}\}$.
Then, $ J$ is $D^{2n}_{k+1}$-module and we have
$$ \Sigma|_{D_{k+1}^{2n}}/ J\simeq \pi_{n-i}|\cdot|^{\frac{i}{2}}\rtimes \sigma|_{ D^{2i}_{k+1}}, \ J\simeq \cind^{D^{2n}_{k+1}}_{w^{-1}Hw\cap D_{k+1}^{2n}}\sigma',$$
where $\pi_{n-i}|\cdot|^{i/2}=(\pi_{n-i}|\cdot|^{i/2},V')$ and $\sigma'=(\sigma',V'\otimes V)$ is defined by
$$\sigma'\left(
\begin{pmatrix}
a&*&b&*\\
&c&&d\\
&&1&\\
&&&1_{k}
\end{pmatrix}
\right)=\pi_{n-i}|\cdot|^{i/2}\left(
\begin{pmatrix}
a&b\\
&1
\end{pmatrix}
\right)
\otimes\sigma
\left(
\begin{pmatrix}
c&d\\
&1_{k}
\end{pmatrix}
\right).
$$
Then, since $\pi_{n-i}|\cdot|^{i/2}|_{D_1^{2n-2i-1}}\simeq \pi_{n-i-1}|\cdot|^{\frac {i+1}{2}}\rtimes I_1$ by Proposition \ref{P:zell}, we have
$$\sigma'\simeq \cind^{w^{-1}Hw\cap D_{k+1}^{2n}}_{G'} \pi_{n-i-1}|\cdot|^{\frac {i+1}{2}}\boxtimes \bm \sigma,$$
where $G'=D^{2n}_{2i+2}\diag(1_{2n-2i-2},G)$ and $ \pi_{n-i-1}|\cdot|^{\frac {i+1}{2}}\boxtimes \bm \sigma$ is regarded as a representation of $G'$ by
$$U_{(2n-2i-2,2i+2)}\backslash G'\simeq \gll_{2n-2i-2}\times G.$$
Thus we have
$$J \simeq \pi_{n-i-1}|\cdot|^{\frac{i+1}{2}}\rtimes \cind^{D_{k+1}^{2i+2}}_G\bm \sigma.$$
\end{proof}
We obtain the following:
\begin{prop}\label{P:fl}
There is a sequence
$$\pi_n=J_1'\supset J_2'\dots\supset J_n'$$
 of subspaces of $\pi_n$ such that
 \begin{itemize}
 \item $J'_i$ is isomorphic to $J_i=J_i(\pi,n):=\pi_{n-i}|\cdot|^{i/2}\rtimes I_i$ as $D^{2n}_i$-representation ($i=1,\dots,n$) and
 \item $J'_i|_{D^{2n}_{i+1}}/J'_{i+1}\simeq K_i=K_i(\pi,n):=\pi_{n-i}|\cdot|^{\frac{i}{2}}\rtimes I_i|_{D^{2i}_{i+1}} \ (i=1,\dots,n-1)$.
 \end{itemize}
\end{prop}
\begin{proof}
We have $\pi_n|_{D^{2n}_1}\simeq J_1$ by proposition \ref{P:zell}.
Then, using Lemma \ref{L:tes} repeatedly, we have this proposition immediately (note that if $\sigma=I_i$, we have $ \cind^{D_{i+1}^{2i+2}}_G\bm \sigma \simeq I_{i+1}$ under the notations in Lemma \ref{L:tes}).
\end{proof}

\subsection{The proof of Proposition \ref{P:uniq}}\label{SS:pf}
Finally, we give the proof of Proposition \ref{P:uniq}.

\begin{proof}[proof of Proposition \ref{P:uniq}]
By Lemma \ref{L:nsc} and \ref{L:n12}, we can assume that $\pi$ is supercuspidal and $n>2$.

By Proposition \ref{P:sh} and \ref{P:fl}, the restriction map
$$\Ho _{\gll_n\times\gll_n}(\pi_n,\tau\times\tau')\map \Ho _{\gll_n\times\gll_n}(J_n',\tau\times\tau')\simeq \Ho _{\gll_n\times\gll_n}(I_n^{\chi_\pi},\tau\times\tau') $$
 is well-defined.
 We show that this map is injective.
Suppose $\Ho _{\gll_n\times\gll_n}(\pi_n/J_n',\tau\times\tau')\neq0$ for the sake of contradiction.
Then, we have $\Ho _{\gll_n}(\pi_n/J_n',\tau)\neq0.$
By Proposition \ref{P:fl}, we have $\Ho _{\gll_n}(K_i,\tau)\neq0$ for some $i\in\{1,\dots,n-1\}$.
Finally, using Proposition \ref{P:zell} and Lemma \ref{L:tes} repeatedly, we have
$$\Ho _{\gll_n}(\pi_{n-j}|\cdot|^{j/2}\rtimes \sigma,\tau)=\Ho _{\gll_n}(\pi_{n-j}|\cdot|^{j/2}* \sigma|_{\gll_{2j-n}},\tau)\neq0$$
for some $j\in\zz_{>0}$ such that $j<n \leq  2j$ and representation $\sigma$ of $D^{2j}_{n}$.
This contradicts $\tau\in\mathrm{Irr_{sc}}$ and $n/2>1$.

Therefore,  we have 
$$\dim_\cc\Ho _{\gll_n\times\gll_n}(I_n^{\chi_\pi},\tau\times\tau')\geq\dim_\cc\Ho _{\gll_n\times\gll_n}(\pi_n,\tau\times\tau').$$
On the other hand, the equality
$$\dim_\cc\Ho _{\gll_n\times\gll_n}(I_n^{\chi_\pi},\tau\times\tau')=\begin{cases}
1 \ \text{ if } \tau'=\tau^\vee\chi_\pi;\\
0 \ \text{ otherwise}
\end{cases}$$
holds since $I_n^{\chi_\pi}|_{\gll_n\times\gll_n}\simeq \cind^{\gll_n\times\gll_n}_{\Delta \gll_n} \chi_\pi(\det)$, where $\Delta \gll_n$ is the diagonal embedding of $\gll_n$ to \mbox{$\gll_n\times\gll_n$}.
Thus we get the required statement. 
\end{proof}

\section{Some remarks for general rank case}\label{S:grank}
By Theorem \ref{T:main1} (ii), we have
$$\mathrm{Hom}_{\gll_n\times\gll_{n}}(\pi_{n},\tau\boxtimes\tau^\vee\chi_\pi)\neq0$$
for any  $\tau\in {\rm Irr}\gll_n$  if $\pi$ is approximately tempered.
To conclude this paper, let us consider the branching laws of the Speh representations with respect to general block diagonal subgroups.

 \begin{theo}\label{T:main2}
 Assume $\pi$ is approximately tempered.
Then,  we have
$${\rm Hom}_{\gll_n\times\gll_{n+2l}}(\pi_{n+l},\tau\boxtimes\tau^\vee\chi_\pi\times\pi_{l})\neq0$$
for any  $\tau\in {\rm Irr}\gll_n$.
 \end{theo}
\begin{proof}
Let $L$ be a nonzero element of ${\rm Hom}_{\gll_n\times\gll_{n}}(\pi_{n}|\cdot|^{\frac{l}{2}},\tau\boxtimes\tau^\vee\chi_\pi|\cdot|^{l})$.
For $f\in \pi_n|\cdot|^{-\frac{l}{2}}\times\pi_l|\cdot|^{\frac{n}{2}}=\pi_n|\cdot|^{\frac{l}{2}}*\pi_l|\cdot|^{-\frac{n}{2}}$,
define $\tilde L f (g)=L\otimes \mathrm{id}_{\pi_l|\cdot|^{-\frac{n}{2}}} (f(g))$ for any $g\in \gll_{n+2l}$.
Then, $\tilde L$ is a nonzero element of 
$${\rm Hom}_{\gll_n\times\gll_{n+2l}}(\pi_n|\cdot|^{-\frac{l}{2}}\times\pi_l|\cdot|^{\frac{n}{2}},\tau\boxtimes\tau^\vee\chi_\pi\times\pi_{l})$$ and the following diagram

\[
  \begin{CD}
     \pi_n|\cdot|^{-\frac{l}{2}}\times\pi_l|\cdot|^{\frac{n}{2}} @>{f\mapsto f(1_{2n+2l})}>>   \pi_n|\cdot|^{\frac{l}{2}}\boxtimes\pi_l|\cdot|^{-\frac{n}{2}}  \\
  @V{\tilde L}VV    @V{L\otimes \mathrm {id}}VV \\
  \tau\boxtimes\tau^\vee\chi_\pi\times\pi_{l}   @>{f'\mapsto f'(1_{n+2l}) }>>  \tau\boxtimes\tau^\vee\chi_\pi|\cdot|^l\boxtimes\pi_{l}|\cdot|^{-\frac{n}{2}}
  \end{CD}
\]
is commutative. 

 Since $\pi$ is approximately tempered, $\pi_{n+l}$ is a subrepresentation of $\pi_n|\cdot|^{-\frac{l}{2}}\times\pi_l|\cdot|^{\frac{n}{2}}$.
Since $\pi_n|\cdot|^{\frac{l}{2}}\boxtimes\pi_l|\cdot|^{-\frac{n}{2}}$ is irreducible,  $\pi_{n+l}$ is mapped onto $\pi_n|\cdot|^{\frac{l}{2}}\boxtimes\pi_l|\cdot|^{-\frac{n}{2}}$ by substituting the unit.
Then, by the above commutative diagram, we have $\tilde L|_{\pi_{n+l}}$ is a nonzero element of ${\rm Hom}_{\gll_n\times\gll_{n+2l}}(\pi_{n+l},\tau\boxtimes\tau^\vee\chi_\pi\times\pi_{l})$.
\end{proof}

Let $\tau'\in\mathrm{Irr}\gll_{n+2l}$.
According to \cite[Proposition 2.3]{ATOBE2018281}, $\tau'$ and $\tau^\vee\chi_\pi\times\pi_l$ have the same cuspidal support  if $\pi$, $\tau$, and $\tau'$ are unramified and $\Ho_{\gll_n\times\gll_{n+2l}}(\pi_{n+l},\tau\boxtimes\tau')\neq0$ (we note that  if $\pi$ and $\tau$ are unramified and unitary, then $\tau^\vee\chi_\pi\times\pi_l$ is unramified and unitary).
This fact is important for determining the near equivalence classes of global Miyawaki lifts (see \cite{itofirst}).
Furthermore, considering Proposition \ref{P:uniq}, it is natural to expect the following conjecture:
\begin{conj}\label{C:end}
Assume $\pi, \tau,$ and  $\tau'$ are unitary.
Then, the following should hold:
\begin{enumerate}
\renewcommand{\labelenumi}{(\roman{enumi})}
\item (uniqueness) $\Ho_{\gll_n\times\gll_{n+2l}}(\pi_{n+l},\tau\boxtimes\tau')\neq0\Rightarrow \tau'=\tau^\vee\chi_\pi\times\pi_l$.
\item (multiplicity at most one) $\dim_\cc {\rm Hom}_{\gll_n\times\gll_{n+2l}}(\pi_{n+l},\tau\boxtimes\tau')\leq1$.
\end{enumerate}
\end{conj}


  \bibliography{aaaaa}
\bibliographystyle{amsalpha} 
\end{document}